\documentclass{article}

\usepackage{PRIMEarxiv}

\usepackage{amsthm}

\newtheorem{theorem}{Theorem}

\newtheorem{corollary}{Corollary}[theorem]

\newtheorem{definition}{Definition}[theorem]

\newtheorem*{remark}{Remark}





\usepackage[utf8]{inputenc} % allow utf-8 input
\usepackage[T1]{fontenc}    % use 8-bit T1 fonts
\usepackage{hyperref}       % hyperlinks
\usepackage{url}            % simple URL typesetting
\usepackage{booktabs}       % professional-quality tables
\usepackage{amsfonts}       % blackboard math symbols
\usepackage{nicefrac}       % compact symbols for 1/2, etc.
\usepackage{microtype}      % microtypography
\usepackage{lipsum}
\usepackage{fancyhdr}       % header
\usepackage{graphicx}       % graphics
\graphicspath{{media/}}     % organize your images and other figures under media/ folder

\usepackage{graphicx}%
\usepackage{multirow}%
\usepackage{amsmath,amssymb,amsfonts}%
\usepackage{mathrsfs}%
\usepackage[title]{appendix}%
\usepackage{xcolor}%
\usepackage{textcomp}%
\usepackage{manyfoot}%
\usepackage{booktabs}%
\usepackage{algorithm}%
\usepackage{algorithmic}%
\usepackage{listings}%
\usepackage{optidef}

\usepackage{geometry}
\usepackage{pdflscape}
\usepackage{todonotes}
\usepackage{accents}

\newcommand{\ubar}[1]{\underaccent{\bar}{#1}}
\newcommand*{\probability}{\mathbb{P}}

\newcommand*{\compon}{i}
\newcommand*{\scen}{s}
\newcommand*{\bbnode}{\iota}
\newcommand*{\treenode}{\eta}
\newcommand*{\alloclevel}{\ell}
\newcommand*{\componoutcome}{\kappa}

\newcommand*{\leaf}{\lambda}

\newcommand*{\pathindex}{\delta}

\newcommand*{\secondstagerhs}{b}
\newcommand*{\secondstagerhsvec}{\textbf{\secondstagerhs}}

\newcommand*{\costalloc}{c}
\newcommand*{\costallocvec}{\textbf{\costalloc}}
\newcommand*{\secondstagecost}{q}
\newcommand*{\secondstagecostvec}{\textbf{\secondstagecost}}
\newcommand*{\capac}[1][]{u_{#1}}
\newcommand*{\capaci}{\capac[\compon]}

\newcommand*{\bigmval}[1][]{M_{#1}}
\newcommand*{\bigmvali}{\bigmval[\compon]}

\newcommand*{\numcompon}{n}
\newcommand*{\numalloclevels}{L}
\newcommand*{\numscens}{S}
\newcommand*{\numoutcomes}{K}
\newcommand*{\numfirststagevars}{n_1}
\newcommand*{\numsecondstagevars}{n_2}
\newcommand*{\numsecondstagemainconstrs}{m}

\newcommand*{\firststagevar}{x}

\newcommand*{\firststagevec}{\textbf{\firststagevar}}
\newcommand*{\firststagegenericvec}{\tilde{\textbf{\firststagevar}}}

\newcommand*{\firststagegenerici}{\tilde{\firststagevar}_\compon}
\newcommand*{\firststageallocvar}{x}
\newcommand*{\firststagealloc}[1][]{\firststageallocvar_{#1}}

\newcommand*{\firststageallocil}{\firststagealloc[\compon\alloclevel]}

\newcommand*{\firststageallocvec}{\textbf{\firststageallocvar}}

\newcommand*{\secondstagevar}{y}
\newcommand*{\secondstagevec}{\textbf{\secondstagevar}}

\newcommand*{\auxvar}{\theta}

\newcommand*{\capdualvar}{\beta}
\newcommand*{\capdual}[1][]{\capdualvar_{#1}}
\newcommand*{\capduali}{\capdual[\compon]}

\newcommand*{\statevar}{\xi}

\newcommand*{\staterealizvar}{\statevar}
\newcommand*{\staterealiz}[2]{\staterealizvar_{#1}^{#2}}
\newcommand*{\staterealizi}{\staterealiz{\compon}{}}

\newcommand*{\staterealizis}{\staterealiz{\compon}{\scen}}
\newcommand*{\otherrandvar}{\gamma}

\newcommand*{\recfn}{\Theta}
\newcommand*{\recfnpenalty}{\tilde{\Theta}}

\newcommand*{\probfn}{\probability}
\newcommand*{\stateprobfn}[2]{f_{#1}(#2)}

\newcommand*{\stateprobfndiscri}{\stateprobfn{\compon}{\staterealizi;\alloclevel}}
\newcommand*{\stateprobfndiscrik}{\stateprobfn{\compon}{\componoutcome;\alloclevel}}
\newcommand*{\stateprobfndiscris}{\stateprobfn{\compon}{\staterealizis;\alloclevel}}

\newcommand*{\ev}{\mathbb{E}}

\newcommand*{\parttree}{\mathcal{T}}
\newcommand*{\parttreeroot}{\mathcal{T}^0}
\newcommand*{\parttreefull}{\mathcal{T}^F}

\newcommand*{\firststagefeasregion}{X}
\newcommand*{\firststagegenericfeasregion}{\tilde{X}}
\newcommand*{\firststageallocfeasregion}{X}
\newcommand*{\randvecsupport}{\Xi}

\newcommand*{\components}{I}
\newcommand*{\fixedcomponents}{F}
\newcommand*{\treenodesonpath}{\mathcal{P}}
\newcommand*{\descendentnodes}{\mathcal{D}}
\newcommand*{\parttreenodes}{\mathcal{N}}
\newcommand*{\parttreeleafnodes}{\mathcal{L}}

\newcommand*{\parttreearcs}{\mathcal{A}}

\newcommand*{\secondstagematrix}{W}

\newcommand*{\alloclevelssum}{\sum_{\alloclevel=0}^\numalloclevels}

\newcommand*{\probprod}{\prod_{\compon\in \components}}
\newcommand*{\scenssum}{\sum_{s=1}^\numscens}
\newcommand*{\outcomessum}{\sum_{\componoutcome=0}^\numoutcomes}

\title{A Successive Refinement for Solving Stochastic Programs with Decision-Dependent Random Capacities
}

\author{
  Hugh Medal, Samuel Affar \\
  Department of Industrial Engineering \\
  University of Tennessee \\
  Knoxville\\
  \texttt{hmedal@utk.edu, saffar@vols.utk.edu} \\
}

\begin{document}
\maketitle

\begin{abstract}
We study a class of two-stage stochastic programs in which the second stage includes a set of components with uncertain capacity, and the expression for the distribution function of the uncertain capacity includes first-stage variables. Thus, this class of problems has the characteristics of a stochastic program with decision-dependent uncertainty. A natural way to formulate this class of problems is to enumerate the scenarios and express the probability of each scenario as a product of the first-stage decision variables; unfortunately, this formulation results in an intractable model with a large number of variable products with high-degree. After identifying structural results related to upper and lower bounds and how to improve these bounds, we present a successive refinement algorithm that successively and dynamically tightens these bounds. Implementing the algorithm within a branch-and-cut method, we report the results of computational experiments that indicate that the successive refinement algorithm significantly outperforms a benchmark approach. Specifically, results show that the algorithm finds an optimal solution before the refined state space become too large.
\end{abstract}

% keywords can be removed

\section{Introduction} \label{intro}
Stochastic programming is a pivotal methodology for optimizing under uncertainty, enabling decision-makers to devise robust plans by creating models that capture various possible outcomes with associated likelihoods. Stochastic Programming with Decision-Dependent Uncertainty (SP-DDU) incorporates the reality that the decisions made at a certain time can influence the uncertainty faced later. Take, for instance, the energy sector, where the strategic planning for generation expansion directly affects the volatility of future electricity prices. Incorporating SP-DDU in this context allows for the probabilities of future price fluctuations to be influenced by present-day renewable energy investments \cite{Zhan.Pinson.2017}. 
In service industries such as airline and hospitality, SP-DDU can be used to determine service overbooking levels, thereby avoiding redundant capacity due to cancellations. The number of customers who show up to receive the service without cancellations depends on how much the service is overbooked \cite{karaesmen2004overbooking}. 

The interaction between decision variables and uncertainties often results in non-convex objective functions and constraints, leading to computational challenges. Thus, solving SP-DDU problems necessitates developing and applying new mathematical techniques to solve these models. In this paper, we focus on a class of SP-DDU that we refer to as stochastic programming with endogenous probabilities (SP-EP), in which decision variables directly alter the distribution of uncertain parameters. Within this category of problems the probability for each scenario is determined by a product of probabilities, with each probability being a function of decision variables. 
 
Specifically, this study provides a solution procedure for a subset of SP-EP problems that we term \textit{stochastic programs with decision-dependent random capacities} (SP-DDRC). In these problems the the goal is to optimize the expected value of an objective that depends on the (random) capacities of each component in the set $\components=\{1,\dots,\numcompon\}$. The decision is to allocate resources among the components, given that the probability distribution of the capacity of a component depends on the amount of resources allocated to the component. Problems in this class include network optimization problems such as allocating protection (attack) resources to arcs to maximize (minimize) the expected maximum flow in a network. This class also includes the problem of allocating protection (attack) resources to facilities to maximize (minimize) the expected coverage.
 
A formulation for this class of problems is:

\begin{mini}|s|
    {\firststagegenericvec\in \firststagegenericfeasregion}{\costallocvec^T\firststagegenericvec + \sum_{\staterealizvar\in \randvecsupport} \probability[\staterealizvar \,|\, \firststagegenericvec]\,\recfn(\staterealizvar),}{\label{mod:first-stage}}{}
\end{mini}

 \noindent where $\costallocvec\in\mathbb{R}^{\numfirststagevars}$, $\firststagegenericvec$ is a vector of first-stage decisions that represent the allocation of a scarce set of resources. This vector has feasible region $\firststagegenericfeasregion$ that incorporates resource scarcity constraints as well as other constraints. The core algorithmic framework that we present in this paper is designed to enhance other method for solving \eqref{mod:first-stage} rather than completely solve \eqref{mod:first-stage} by itself. Thus, our framework does not itself require particular structure in the set $\firststagegenericfeasregion$ such as convexity. However, if this framework is used to solve a problem for complicated, e.g., non-convex, feasible regions $\firststagegenericfeasregion$, it must be used within a method that can handle such a complicated $\firststagegenericfeasregion$. That said, the test problem that we consider in our computational study does have a convex set $\firststagegenericfeasregion$.
 
 In this work we focus on problems in which the second-stage problem contains a set of components $1\dots,\numcompon$, each having a discrete random capacity level $\staterealizi\in\{0,1,\dots,\numoutcomes\}$. In network problems components could represent nodes or arcs, and in facility location problems components could represent facilities. So, $\staterealizvar$ is the realization of a random vector representing the capacity levels of each component, and this vector has support $\randvecsupport$. 
 
 The term $\probability[\staterealizvar \,|\, \firststagegenericvec]$ is the probability of realization $\staterealizvar$ conditional on the first-stage decisions $\firststagegenericvec$, which implies that the uncertainty in the realization $\staterealizvar$ depends on the first-stage decisions; thus, this formulation is a stochastic program with decision-dependent uncertainty. 
 
 A key assumption in our study is that the random component capacities are mutually independent. Letting $\stateprobfn{\compon}{\staterealizi;\firststagegenericvec}$ be the probability that component $\compon$ is in state $\staterealizi$ given resource allocation $\firststagegenericvec$, the probability of realization $\staterealizvar$ is:

 \begin{equation}
	\probfn[\staterealizvar \,|\,\firststagegenericvec] =\prod_{\compon=1}^ \numcompon \stateprobfn{\compon}{\staterealizi;\firststagegenericvec}. \label{probExpr}
\end{equation}

\noindent While there are problems with decision-dependent random capacities for which independence does not hold, this independence assumption still applies to many important problems. In addition, as we will describe in our discussion of possible extensions (see \S\ref{conclusion}), it is possible to model correlated component capacities within our framework if the random capacities are \textit{conditionally independent}, i.e., the capacities are mutually independent conditional upon some event (see also \cite{Medal.2015}). For instance, in the case that the capacities are degraded due to a weather event, it could be reasonable to assume the post-event capacities of components are mutually independent, conditional upon the location and severity of the weather event.

The recourse function $\recfn$ in \eqref{mod:first-stage} only depends on the realization vector $\staterealizvar$, unlike classic stochastic programming formulations in which this function also depends on the first-stage decisions. This recourse function is formulated as:

\begin{mini!}|s|
    {}{\secondstagecostvec^T\secondstagevec \label{eq:obj-recourse}}{\label{mod:recourse}}{\recfn(\staterealizvar) =}
    \addConstraint{\secondstagematrix\secondstagevec}{= \secondstagerhsvec \label{constr:recourse}}
    \addConstraint{\secondstagevar_\compon}{\leq \capaci\,\staterealizi}{\quad\forall\compon=1,\dots,\numcompon\quad [\beta_i]\label{constr:capacity}}
    \addConstraint{\secondstagevec}{\geq 0.}{\label{bound}}
\end{mini!}

\noindent where $\secondstagecostvec\in\mathbb{R}^{\numsecondstagevars}$, $\secondstagevec\in\mathbb{R}^{\numsecondstagevars}$ is a vector of second-stage decisions, $\secondstagematrix\in\mathbb{R}^{\numsecondstagemainconstrs\times\numsecondstagevars}$, and $\secondstagerhsvec\in\mathbb{R}^{\numsecondstagemainconstrs}$. Constraints \eqref{constr:capacity} represent the fact that the random capacity of each component $\compon$ in the system depends on the value of the random variable $\staterealizi$. The $\capduali$ variables in brackets represent the dual multipliers of these constraints.

Our approach works best for problems in which a tight upper bound can be found for the dual multipliers $\capduali$. Recourse problems for which this is the case include the maximum flow problem and the maximum coverage assignment problem with capacitated facilities; for these problems $\capdual\leq 1$.

Solving \eqref{mod:first-stage} is challenging for several reasons. First, modeling the probability of each realization $\staterealizvar$, $\probability[\staterealizvar \,|\, \firststagegenericvec]$, will typically require non-convex multilinear terms resulting from multiplying the $\stateprobfn{\compon}{\staterealizi;\firststagegenericvec}$ terms over all $i=1,\dots,\numcompon$. These multilinear terms are expanded even more by the fact that formulating the deterministic equivalent form of \eqref{mod:first-stage} results in an objective function in which the multilinear $\probability[\staterealizvar \,|\, \firststagegenericvec]$ term for a realization $\staterealizvar$ is also multiplied by the second-stage objective \eqref{eq:obj-recourse} for that realization. Although linearizing these multilinear terms is easier if the $\firststagegenericvec$ is assumed to be a binary vector (which we assume in this paper), this linearization still requires adding many additional constraints and auxiliary variables to iteratively form McCormick envelopes. While these envelopes are exact for binary solutions, the linear programming relaxation can be quite weak, as we found in our computational experimentation.

 While prior research in SP-EPs has solved approximate versions of \eqref{mod:first-stage} or solved it heuristically, this study represents the first exact approach for solving this problem. We introduce a successive refinement algorithm (SRA), which accurately captures these probability functions while overcoming the computational challenges. This method enables us to derive tighter bounds for two-stage stochastic programs with endogenous probabilities. We validate the efficiency of our approach through extensive testing on various test problems and benchmark comparisons.
 
 \section{Related Work}
 \label{related-work}
 Research on stochastic programming with DDU can be classified according to decisions' effect on uncertainty \cite{Goel.Grossmann.2006}. In the first type, the decisions affect the timing that uncertainty is realized in multi-stage stochastic programming. Specifically, the non-anticipativity constraints (NACs) can be negated by activating a variable via disjunctive constraints \cite{Goel.Mulkay.2006,Colvin.Maravelias.2009}. Thus, we refer to this area of research as stochastic programming with endogenous uncertainty realization (SP-EUR). In the second type, which pertains to Distributionally Robust Optimization Models, the decisions have influenced the ambiguity and distribution of uncertainty sets \cite{Luo.Mehrotra.2020}. This category is called Robust Optimization with Decision-Dependent Ambiguity (RO-DDA). In the third type,  which is the focus of this study, the decisions alter the probability distribution governing the uncertainty. Hence, we term this area stochastic programming with endogenous probabilities (SP-EP). SP-EP has received less attention in the literature than SP-EUR, perhaps because it introduces significant non-convexity and finding an explicit and tractable representation of the objective function is difficult.

The literature uses various approaches for SP-EPs, including approximations and decoupling techniques, to avoid nonlinearity before solving the problem using a heuristic or sampling approach. Some studies apply decoupling techniques for the objective to avoid non-convexity and use a sampling technique or an heuristic to solve the resulting model, which usually has a larger sample space. This approach is common in power grid planning and distribution \cite{Ma.Guo.2017,Zhang.Cao.2023}. Ma et al. \cite{Ma.Guo.2017} decoupled the decision-dependence between power distribution resilient-oriented design by decoupling the line failure status into two independent parameters representing separate component outcomes (denoting line status with or without hardening) and expanding the scenario space. The authors used a progressive hedging algorithm to solve the simplified model. Zhang et al. \cite{Zhang.Cao.2023} used the same decoupling approach for a similar problem where hardening decisions are made to make the power grid more resilient but utilized Sample Average Algorithm (SAA) to solve the resulting stochastic MIP. Similarly, Yin et al. \cite{Yin.Hou.2023q6f} approached renewable energy expansion planning using a decision-dependent stochastic model that factors the co-dependence relationship between wind farm capacity expansion decisions and wind power uncertainties. An affine function based on historical data was used to show the relationship between wind power output and the expansion decision. To handle scenario generation,  a scenario set of the power wind output was generated for every possible expansion decision, enlarging the scenario sample space, and the resulting model was solved using a progressive hedging algorithm. Hu et al. \cite{Hu.Peng.2020} proposed a decision-dependent model that considers the relationship between operational decision variables and force outage rates of electrical components. The proposed model can be decoupled into stages, making it suitable for an Adaptive reliability improvement unit commitment algorithm (ARIUC). It refines the decision variable values for each model stage based on new information from the other stage. These decoupling techniques involve enlarging the sample space by introducing new parameters and increasing the number of scenarios, making the resulting decision-independent models less scalable. 

Another type of approach for solving SP-EPs is to discretize the first-stage variables and then linearize the multilinear terms using a technique called a probability chain \cite{Losada.2012,OHanley.2013}. Medal et al. \cite{Medal.2015} studied a two-stage facility location stochastic program with endogenous uncertainty. Each facility's post-disruption capacity level depends on the protection allocation level assigned to the facility in the first stage. Because of the decision-dependent uncertainty, the model's first stage is independent of the second stage, such that for a fixed first-stage solution, the optimal second-stage cost can be computed. The models is  linearized using a sequential set of constraints to linearize and calculate the probability-weighted cost for each scenario. This linearized reformulation allows the problem to be recast into a mixed-integer program that can be solved exactly. Zhou et al. \cite{Zhou.2022}  utilized the same approach for a supplier development program selection problem where supplier performance depends on the decision maker's level of development investment in the supplier. The same technique is used in \cite{Bhuiyan.2019} to solve a fuel incentivization reduction program to reduce wildfires where a landowner's acceptance probability depends on the offer amount decided in the first stage. Similarly, Bhuiyan et al. \cite{Bhuiyan.2020} utilize the same recursive linearization technique for a reliable network design problem and an accelerated L-shaped decomposition algorithm to solve the resulting two-stage stochastic MILP. However, the recursive linearization technique increases the size of the resulting MIPs significantly, making the approach less scalable. 

Some studies have also used tractable (e.g., convex) approximations to resolve the decision dependence in SP-EPs. Yin et al. \cite{Yin.Hou.2023} studied the same problem in \cite{Yin.Hou.2023q6f} with but reduced the scenario sample space using a Point Estimate Method (PEM) based model, which utilizes special point concentrations instead of the entire probability density function, approximating the expectation of the second stage objective across the uncertain parameter. An iterative algorithm that utilizes Bender's decomposition in each iteration solves the model. Peeta et al. \cite{Peeta.Viswanath.2010} studied a network protection problem as an SPEU, approximating the objective function using a Taylor series expansion, which resolves the non-convexity caused by SPEU. This approximate model was then solved using sample average approximation. Because they approximate the objective function, their approach does not guarantee optimality or provide a performance guarantee. 

Other studies focus on finding locally, rather than globally, optimal solutions. Karaesmena and Ryzin \cite{karaesmen2004overbooking} studied the overbooking problem for multiple inventory classes and reservations where the number of customers who show up depends on the overbooking level in the first stage, which was assuming to be a continuous quantity. A locally optimal method was proposed that uses a stochastic gradient descent algorithm along with a gradient estimation scheme and a projection step.

Other studies have developed heuristic algorithms to solve SP-EPs. Held et al. \cite{Held.Woodruff.2005} studied a multi-stage stochastic network interdiction problem in which the first-stage decisions determine the probabilities of the random scenarios, using a heuristic to solve the problem. Relaxing assumptions made in \cite{Peeta.Viswanath.2010}, \cite{Du.Peeta.2014} used an iterative heuristic algorithm in conjunction with Monte Carlo simulation.

SP-DDRC is an important class of problems in stochastic programming literature several important applications. A good example is the well-studied stochastic network interdiction model, which has applications in cyber-physical systems security \cite{10.1016/j.ejor.2019.06.024} and interdicting illicit supply chain networks \cite{10.1016/j.cie.2021.107708}. Similarly, stochastic network protection is deployed for critical infrastructure protection planning \cite{10.1016/j.cor.2006.09.019}. Different variations of SP-DDRC problems have been studied, but none incorporate endogenous probabilities. Cormican et al. \cite{Cormican.98} provided the seminar work on stochastic network interdiction problem, studying a maximum flow network interdiction problem in which interdiction attempts on arcs have a probability of success. This work differs from ours in that, while their study consider fixed values of the interdiction success probability, in SP-DDRC problems the random capacity of components has a probability distribution that depends on the first-stage variables. That said, the successive refinement algorithm that we present in this study is inspired by the sequential approximation algorithm provided by \cite{Cormican.98}.

In summary, there is limited research on stochastic programming with endogenous probabilities (SP-EPs) compared to stochastic programming with endogenous uncertainty realization and robust optimization with decision-dependent ambiguity sets. Moreover, no study incorporates endogenous probabilities into SP-DDRC problems. The existing literature on SP-EPs has largely avoided the issue of high nonlinearity by using decoupling, recursive nonlinearization, or approximations that do not scale properly or fail to represent the interdependency in SP-EPs. These issues highlight the need for scalable and efficient solution techniques for SP-DDRC problems.

\section{Problem Description and Benchmark Method}\label{prob-desc}
Before describing our method, we first present a formulation of \eqref{mod:first-stage} as a multilinear program that can be linearized straightforwardly and solved via a commercial solver such as Gurobi. 

In the remainder of this paper we make two additional assumptions. First, for ease of exposition, we assume that $\numfirststagevars=\numcompon$ so that there is a one-to-one and onto mapping between first stage variables in the vector $\firststagegenericvec$ and components $\compon=1,\dots,\numcompon$. Thus, $\probfn[\staterealizvar \,|\,\firststagegenericvec] =\prod_{\compon=1}^ \numcompon \stateprobfn{\compon}{\staterealizi;\firststagegenerici}$. This assumption is straightforward to relax. 

Second, we assume that the vector $\firststagegenericvec$ is integer-valued and each variable $\firststagegenericvec\in\{0,1,\dots,\numalloclevels\}$. Under this assumption, $\firststagegenericvec$ can be represented as the sum of binary variables multiplied by coefficients, i.e., $\firststagegenerici=\alloclevelssum \alloclevel \firststageallocil$, where $\firststageallocil$ is a binary variable that is $1$ if $\alloclevel$ units are allocated to component $\compon$ and $0$ otherwise. Given this assumption, we now use the notation $\firststageallocvec=\left(\firststageallocil\right)_{i=1,\dots,\numcompon,\ell=0,\dots,\numalloclevels}$ to denote the binary version of the vector of first-stage decisions and the notation $\firststageallocfeasregion$ to denote its feasible region. Thus, the probability of a realization can be represented as:

\begin{equation}
	\probfn[\staterealizvar \,|\,\firststageallocvec] =\probprod \alloclevelssum \stateprobfndiscri \firststageallocil, \label{probExpr-discrete}
\end{equation}

\noindent given that a constraint is included in $\firststagefeasregion$ that requires exactly one allocation level for each component, i.e., $\alloclevelssum \firststageallocil$ for all $\compon=1,\dots,\numcompon$. When the variables in $\firststageallocvec$ are binary, the  McCormick envelopes used to linearize products of variables are exact for binary solutions. Thus, this assumption allows us to use our method within modern branch-and-cut solvers such as Gurobi.

Given the aforementioned assumptions, problem \eqref{mod:first-stage} can be formulated as the following deterministic equivalent stochastic program with a multilinear objective. First, we enumerate the realizations of $\staterealizvar$ and index them by scenario, i.e., $\staterealizvar^1,\dots,\staterealizvar^\numscens$ and make copies of $\secondstagevec$ for each scenarios $\scen$. Then the formulation is:

\begin{mini!}|s|
    {\firststagevec\in \firststagefeasregion}{\costallocvec^T\firststageallocvec + \scenssum \left( \probprod \alloclevelssum \stateprobfndiscris \firststageallocil\right) \secondstagecostvec^T\secondstagevec^\scen \label{eq:obj-multi}}{\label{mod:multi-program}}{}
    \addConstraint{\secondstagematrix\secondstagevec^\scen}{= \secondstagerhsvec}{\quad\forall\scen=1,\dots\numscens\label{constr:recourse-multi}}
    \addConstraint{\secondstagevar_\compon^\scen}{\leq \capaci\,\staterealizis}{\quad\forall\compon=1,\dots,\numcompon; \scen=1,\dots\numscens \label{constr:capacity-multi}}
    \addConstraint{\secondstagevec^\scen}{\geq 0}{\quad \forall\scen=1,\dots\numscens.\label{bound-multi}}
\end{mini!}

Unfortunately, this deterministic equivalent model \eqref{mod:multi-program} is challenging to solve because of the multilinear terms in the objective \eqref{eq:obj-multi}, even though the $\firststageallocil$ variables are binary. Our computational testing (see \S\ref{results}) found that this formulation is only able to solve small problems. Thus, in \S\ref{struct-results} we describe structural properties present in the formulation \eqref{mod:first-stage} and then in \S\ref{sra} we leverage these properties to develop a method that alleviates some of the challenges with solving this deterministic equivalent formulation.

\section{Structural Results}\label{struct-results}

\subsection{Elementary Lower and Upper Bounds}
We first observe that by elementary linear programming theory, the recourse function $\recfn$ is convex over $\randvecsupport$ and thus, for a fixed $\firststageallocvec$, a \textit{mean value recourse problem} yields a lower bound on the true objective value for $\firststageallocvec$ due to Jensen's inequality, i.e.,

\begin{equation} \label{eqn:lb}
    \recfn(\bar{\staterealizvar}(\firststageallocvec)) \leq \ev[\recfn(\staterealizvar)\, |\,\firststageallocvec ],
\end{equation}

\noindent where $\bar{\staterealizvar}(\firststageallocvec) = \ev[\staterealizvar\,|\,\firststageallocvec]$. The mean value recourse problem is

\begin{mini!}|s|
    {}{\secondstagecostvec^T\secondstagevec \label{eq:mvp}}{\label{mod:mvp-recourse}}{\recfn(\bar{\staterealizvar}(\firststageallocvec)) =}
    \addConstraint{\secondstagematrix\secondstagevec}{= \secondstagerhsvec \label{constr:mvp-recourse}}
    \addConstraint{\secondstagevar_\compon}{\leq \capaci\,\bar{\staterealizi}(\firststageallocvec)}{\quad\forall\compon=1,\dots,\numcompon\label{constr:mvp-capacity}}
    \addConstraint{\secondstagevec}{\geq 0,}{\label{constr:mvp-bound}}
\end{mini!}

\noindent and $\ev[\recfn(\staterealizvar)\, |\,\firststageallocvec ] = \sum_{\staterealizvar\in \randvecsupport} \probability[\staterealizvar \,|\, \firststagevec]\,\recfn(\staterealizvar)$.

Next, to obtain an upper bound, we define a diagonal matrix $\mathbf{M}=\textbf{diag}(\{\bigmvali\, : \, i\in \components\})$ where $\bigmvali$ is an upper bound on the dual multipliers $\capduali$ in \eqref{mod:recourse}. Then, we observe that for problems of the form \eqref{mod:recourse} there exists the following \textit{penalized reformulation}:

\begin{mini!}|s|
    {}{\secondstagecostvec^T\secondstagevec + \staterealizvar^T \mathbf{M} \secondstagevec \label{eq:penalty-mvp}}{\label{mod:penalty-mvp}}{\recfnpenalty(\staterealizvar) =}
    \addConstraint{\secondstagematrix\secondstagevec}{= \secondstagerhsvec}{\label{constr:penalty-mvp:recourse}}
    \addConstraint{\secondstagevec}{\geq 0,}{\label{constr:penalty-mvp:bound}}
\end{mini!}

\noindent such that $\recfnpenalty(\staterealizvar) = \recfn(\staterealizvar)$ over $\randvecsupport$ and $\recfnpenalty$ is \textit{concave} over $\randvecsupport$ (see \cite{10.1016/j.ejor.2019.06.024,Cormican.98}). Thus, for a fixed $\firststageallocvec$ we again apply Jensen's inequality to obtain an upper bound:

\begin{equation} \label{eqn:ub}
\ev[\recfnpenalty(\staterealizvar)\,|\,\firststageallocvec] \leq \recfnpenalty(\bar{\staterealizvar}(\firststageallocvec)).
\end{equation}

Foreshadowing our discussion of the successive refinement algorithm described in the next section, the algorithm works best when $\bigmvali$ is a tight upper bound on $\capduali$. This is the case for recourse problems such as the maximum flow problem or the maximum coverage assignment problem with capacitated facilities; for both of these problems $\capdual\leq 1$.

Although these lower and upper bounds (\eqref{eqn:lb} and \eqref{eqn:ub}) may be useful to gain intuition about the properties of problem \eqref{mod:first-stage}, they may be loose and thus have limited utility within a solution method. In the next section we investigate an approach for improving these bounds.

\subsection{Improving the Lower and Upper Bounds via Support Refinement}
The elementary lower and upper bounds (\eqref{eqn:lb} and \eqref{eqn:ub}) can be extended to a partition of the support $\randvecsupport$. Cormican et al. \cite{Cormican.98} used this idea for a stochastic network interdiction problem that did not explicitly consider decision-dependent uncertainty. Rather than using the concept of a partition of $\randvecsupport$ to describe our method, we use the equivalent concept of a \textit{partition tree} to make the exposition easier to understand, specifically our presentation of the mean value lower bound problem \eqref{mod:epi}. Let $\parttree$ denote a partition tree with nodes set $\parttreenodes(\parttree)$ and arcs set $\parttreearcs(\parttree)$ such that each node represents a partial realization of all components $\compon\in\components$. Thus, a partition tree $\parttree$ corresponds to a partition of the support $\randvecsupport$, and a leaf $\leaf$ in the tree represents a cell in the corresponding partition. 

At each node the random capacity levels are \textit{fixed} for a subset of components $\fixedcomponents(\treenode)\subseteq\{1,\dots,\numcompon\}$. Thus, each node in $\parttreenodes(\parttree)$ is associated with the probability of the realization of the capacities of the components in $\fixedcomponents(\treenode)$ given the resource allocation solution. Specifically, for a given $\parttree$ the probability of node $\treenode\in \parttreenodes(\parttree)$ given allocation vector $\firststageallocvec$ is:

\begin{equation} \label{eqn:nodeprob}
    \mathbb{P}[\treenode\,|\,\firststageallocvec] = \prod_{\compon\in \fixedcomponents(\treenode)} \alloclevelssum \stateprobfn{i}{\componoutcome_{\compon\treenode};\alloclevel} \firststageallocil.
\end{equation}

For a given tree $\parttree$, each leaf node $\leaf$ in the set of leaf nodes $\parttreeleafnodes(\parttree)$ also has a probability (computed via \eqref{eqn:nodeprob}). Moreover, at each leaf the expectation of the random capacity levels is taken over the \textit{unfixed} components, i.e., the set $\components\setminus\fixedcomponents(\leaf)$. For a fixed $\firststageallocvec$, let 
$\bar{\staterealizvar}_{\leaf}(\firststageallocvec) = \left(\ev[\staterealizi\,|\,\firststageallocvec; \leaf]\right)_{\compon=1}^\numcompon$ be the vector of expected capacity levels for the unfixed components, where 

\[\ev[\staterealizi\,|\,\firststageallocvec; \leaf] = \begin{cases}
\componoutcome_{\compon\leaf}, & \compon\in \fixedcomponents(\leaf)\\
\outcomessum \alloclevelssum \stateprobfndiscrik \firststageallocil, & \compon\in \components\setminus\fixedcomponents(\leaf)
\end{cases}\]

\noindent and where $\componoutcome_{\compon\leaf}$ is the (fixed) state of component $\compon$ at leaf node $\leaf$.

A foundational result for our solution method is that any partition tree provides a lower bound on the objective value for any solution $\firststageallocvec$.

\begin{theorem} \label{thm:lb}
For a fixed $\firststageallocvec$ and partition tree $\parttree$:

\begin{equation}
    \ubar{w}(\firststageallocvec;\parttree) = \sum_{\leaf\in \parttreeleafnodes(\parttree)} \mathbb{P}[\leaf\,|\,\firststageallocvec]\, \recfn(\bar{\staterealizvar}_{\leaf}(\firststageallocvec)) \leq \ev[\recfn(\staterealizvar)\, |\,\firststageallocvec].
\end{equation}
\end{theorem}

\begin{proof}

Let $\ev[\recfn(\staterealizvar)\, |\,\firststageallocvec,\leaf]$ denote the expectation of $\recfn(\staterealizvar)$, given allocation $\firststageallocvec$ and the realization of the capacity levels of the fixed nodes at leaf node $\leaf$. Using the conditional form of Jensen's inequality (see \cite{huang1977bounds}), for a given leaf node $\leaf$ we have that $\recfn(\bar{\staterealizvar}_{\leaf}(\firststageallocvec))\leq \ev[\recfn(\staterealizvar)\, |\,\firststageallocvec,\leaf]$ for all $\leaf\in \parttreeleafnodes(\parttree)$. Multiplying both sides by $\mathbb{P}[\leaf\,|\,\firststageallocvec]$ and summing each side over all leaf nodes yields the result.
\end{proof}

At this point we investigate two extreme cases of partition trees, the tree consisting of only the root node and the full partition tree. Letting $\parttreeroot$ denote the partition tree consisting of only the root node, we observe that this solution provides the elementary lower bound \eqref{eqn:lb}. 

\begin{remark}
In this case the root node is also the only leaf node $\leaf^0$ in the set $\parttreeleafnodes(\parttreeroot)$. Thus, $\mathbb{P}[\leaf^0\,|\,\firststageallocvec] = 1$ and, because all components are unfixed at leaf $\leaf^0$, $\bar{\staterealizvar}_{\leaf}(\firststageallocvec)=\bar{\staterealizvar}(\firststageallocvec)$. Hence:

\begin{equation}
    \ubar{w}(\firststageallocvec;\parttreeroot) = \sum_{\leaf\in \parttreeleafnodes(\parttreeroot)} \mathbb{P}[\leaf\,|\,\firststageallocvec] \recfn(\bar{\staterealizvar}_{\leaf}(\firststageallocvec)) = \recfn(\bar{\staterealizvar}(\firststageallocvec))
\end{equation}
\end{remark}

Also, the \textit{full partition tree}, which corresponds to complete partition of $\randvecsupport$, recovers the true objective value of a solution $\firststageallocvec$.

\begin{corollary}\label{cor:full}
Let $\parttreefull$ denote the full partition tree. Then:

\begin{equation}
    \ubar{w}(\firststageallocvec;\parttreefull) = \sum_{\leaf\in \parttreeleafnodes(\parttreefull)} \mathbb{P}[\leaf\,|\,\firststageallocvec] \recfn(\bar{\staterealizvar}_{\leaf}(\firststageallocvec)) = \ev[\recfn(\staterealizvar)\, |\,\firststageallocvec ] = \sum_{\leaf\in \parttreeleafnodes(\parttreefull)} \mathbb{P}[\leaf\,|\,\firststageallocvec] \recfnpenalty(\bar{\staterealizvar}_{\leaf}(\firststageallocvec)) = \bar{w}(\firststageallocvec;\parttreefull)
\end{equation}
\end{corollary}

\begin{proof}
To establish the second equality, not that there is a one-to-one mapping between the leaf nodes in $\parttreeleafnodes(\parttreefull)$ and the set of realizations $\randvecsupport$. Also, because at each leaf node all of the components are fixed, $\bar{\staterealizvar}_{\leaf}(\firststageallocvec)=\staterealizvar$ for all $\leaf\in\parttreeleafnodes(\parttreefull)$. Hence:

\begin{equation}
    \sum_{\leaf\in \parttreeleafnodes(\parttreefull)} \mathbb{P}[\leaf\,|\,\firststageallocvec]\, \recfn(\bar{\staterealizvar}_{\leaf}(\firststageallocvec)) = \sum_{\staterealizvar\in \randvecsupport} \probability[\staterealizvar \,|\, \firststagevec]\,\recfn(\staterealizvar) = 
    \ev[\recfn(\staterealizvar)\, |\,\firststageallocvec ]
\end{equation}

\end{proof}

Figure \ref{fig:1} shows an example of a full refinement tree for a network problem consisting of three arcs (components): (1,2), (2,3), and (3,4). Each arc has a binary random capacity level, i.e., $\staterealizi \in \{0, 1\}$ for $i=1,2,3$. After the root node, the first level of nodes corresponds to fixing arc $(1,2)$ at capacity level 0 and 1, respectively. The second level corresponds to fixing arc $(2,3)$, and the third level $(3,4)$. Thus, there are eight leaf nodes at the bottom, each corresponding to a realization of $\staterealizvar$. Note that in this full refinement tree there are no unfixed components at the leaf nodes.

\begin{figure}
\centering
  \includegraphics[width=0.75\linewidth]{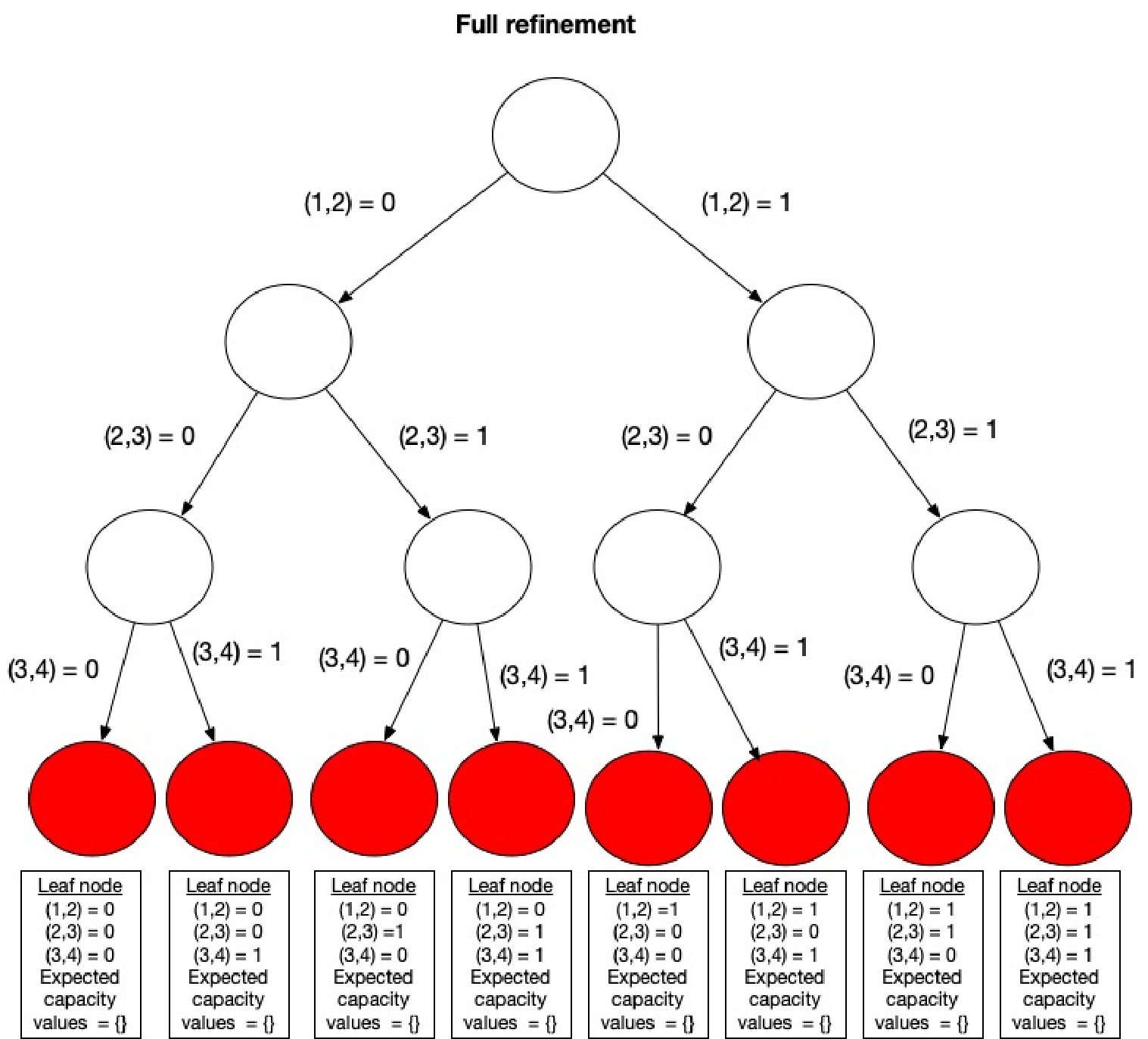}
\caption{ A full refinement three for a network with three arcs. The nodes colored in red are leaf nodes}
\label{fig:1}
\end{figure}

Because the root partition tree $\parttreeroot$ is likely to provide a weak bound and computing the bound for the full partition tree $\parttreefull$ is computationally intractable for large problems, it is desirable to find an intermediate tree that provides a strong bound and yet is much smaller than the full partition tree. Thus, at this point we introduce the concept of refining a tree to obtain a super tree with tighter bounds.

\begin{definition}\label{def:refine}
    A partition tree $\parttree'$ is said to be a \textit{super tree} of a partition tree $\parttree$ if for a particular leaf $\leaf'\in\parttreeleafnodes(\parttree)$ $\parttreenodes(\parttree')= \parttreenodes(\parttree)\cup\descendentnodes(\leaf)$, where $\descendentnodes(\leaf)$ is a set of descendants of leaf $\leaf'$. Furthermore, and $\parttreeleafnodes(\parttree')= (\parttreeleafnodes(\parttree)\setminus\{\leaf'\})\cup\descendentnodes(\leaf)$. Finally,
    $\parttreearcs(\parttree')=\parttreearcs(\parttree)\cup\{(\leaf,\treenode)\,:\,\treenode\in \descendentnodes(\leaf')\}$.
\end{definition}

Thus, the partition of $\randvecsupport$ corresponding to the super tree $\parttree'$ is a refinement of the partition corresponding to $\parttree$. A partially refined tree is shown in Figure \ref{fig:2}. This tree is a super tree of the tree consisting formed by starting with the root node and then adding descendent nodes corresponding to setting the capacity level of $(1,2)$ to 0 and 1, respectively. 

\begin{figure}
\centering
  \includegraphics[width=0.75\linewidth]{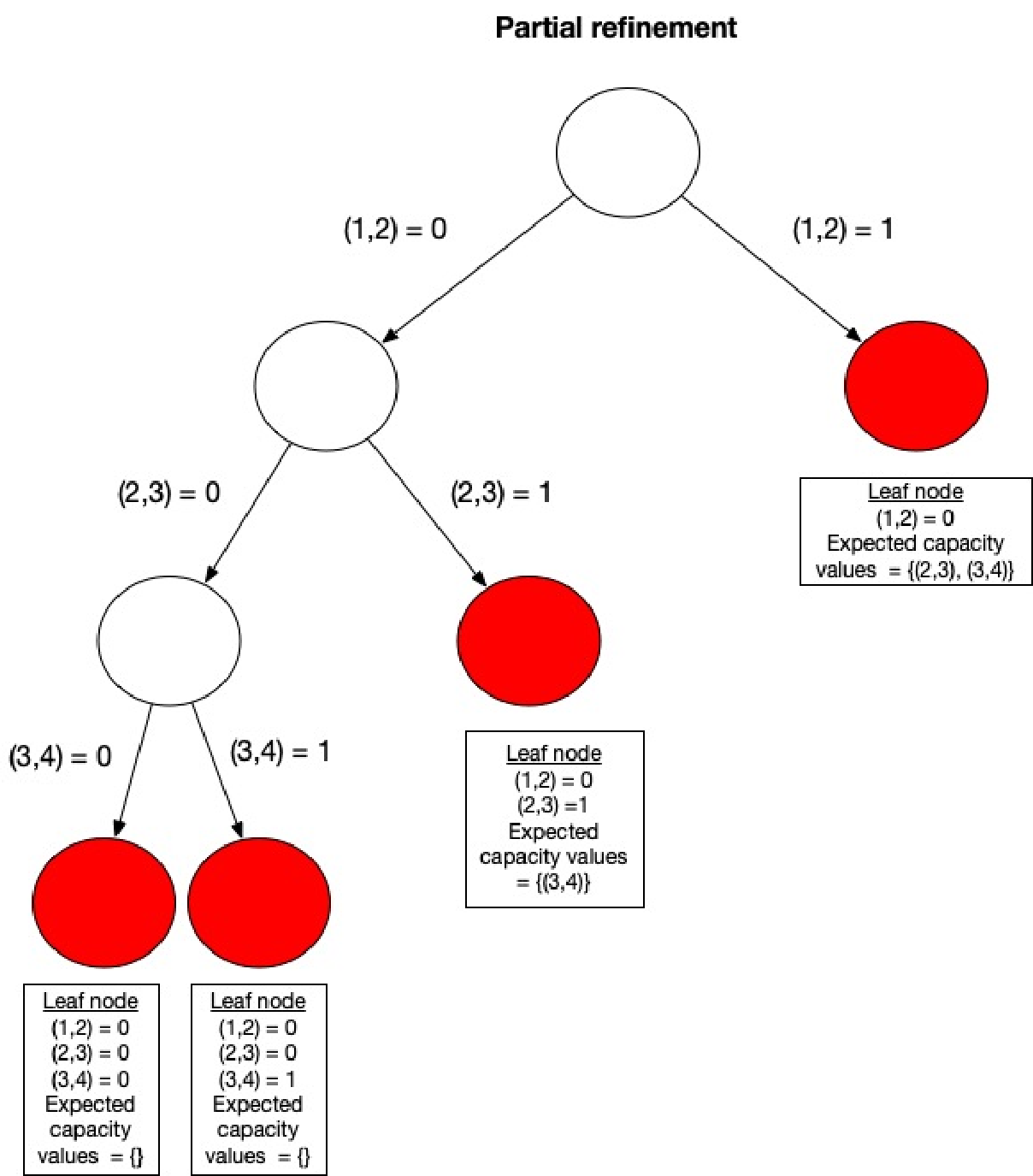}
\caption{A partially refined version of the tree in Figure \ref{fig:1}}
\label{fig:2}
\end{figure}

The following theorem shows that refining a tree yields a lower bound that is least as tight as the lower bound for the original tree.

\begin{theorem} \label{thm:lb-refine}
For a fixed $\firststageallocvec$ and for $\parttree$ and $\parttree'$ such that $\parttree'$ is a super tree of $\parttree$:

\begin{equation}
    \ubar{w}(\firststageallocvec;\parttree) \leq \ubar{w}(\firststageallocvec;\parttree').
\end{equation}
\end{theorem}
\begin{proof}

For a given partition tree $\parttree$, $\ubar{w}(\firststageallocvec;\parttree) = \sum_{\leaf\in \parttreeleafnodes(\parttree)} \mathbb{P}[\leaf\,|\,\firststageallocvec]\, \recfn(\bar{\staterealizvar}_{\leaf}(\firststageallocvec))$. A super tree $\parttree^S$ is formed by choosing a leaf $\leaf\in\parttreeleafnodes(\parttree)$ and replacing it with a set of descendants $\descendentnodes(\leaf)$ (Definition \ref{def:refine}). By the convexity of $\recfn$, this replacement results in a tighter lower bound. This refinement process can be continued until the final super tree $\parttree'$ is attained.

\end{proof}

A straightforward corollary to Theorem \ref{thm:lb} is that any partition tree also provides an upper bound on the objective value for any solution $\firststageallocvec$ by using the penalized reformulation.

\begin{corollary}\label{cor:ub}
For a fixed $\firststageallocvec$ and partition tree $\parttree$:

\begin{equation}
\ev[\recfn(\staterealizvar)\, |\,\firststageallocvec ] \leq \sum_{\leaf\in \parttreeleafnodes(\parttree)} \mathbb{P}[\leaf\,|\,\firststageallocvec]\, \recfnpenalty(\bar{\staterealizvar}_{\leaf}(\firststageallocvec)) = \bar{w}(\firststageallocvec;\parttree),
\end{equation}
\end{corollary}

A straightforward corollary to Theorem \ref{thm:lb-refine} is that refining a tree yields an upper bound that is at least as tight.

\begin{corollary} \label{cor:ub-refine}
For a fixed $\firststageallocvec$ and for $\parttree$ and $\parttree'$ such that $\parttree'$ is a super tree of $\parttree$:

\begin{equation}
    \bar{w}(\firststageallocvec;\parttree) \geq \bar{w}(\firststageallocvec;\parttree').
\end{equation}
\end{corollary}

Let $\ubar{w}(\parttree)$ be the lower bound that is the optimal objective value of the following \textit{mean value lower bound problem for partition} $\parttree$:

\begin{mini}|s|
    {\firststagevec\in \firststagefeasregion}{\costallocvec^T\firststagevec + \sum_{\leaf\in \parttreeleafnodes(\parttree)} \mathbb{P}[\leaf\,|\,\firststageallocvec]\, \recfn(\bar{\staterealizvar}_{\leaf}(\firststageallocvec))}{\label{mod:mvp-part}}{}
\end{mini}

Letting $z^*$ denote the optimal objective value of \eqref{mod:first-stage}, the following corollary follows from Theorems \ref{thm:lb} and \ref{thm:lb-refine}.

\begin{corollary}\label{cor:lb}
For partition trees $\parttree$ and $\parttree'$ such that $\parttree'$ is a refinement of $\parttree$:
\begin{equation}
    \ubar{w}(\parttree) \leq \ubar{w}(\parttree')\leq \ubar{w}(\parttreefull) = z^*.
\end{equation}
\end{corollary}

Corollary \ref{cor:lb} provides a basis for our successive refinement algorithm, which we describe in the next section.

\section{Successive Refinement Algorithm}\label{sra}

The successive refinement algorithm works by starting with the root partition tree $\parttreeroot$ and then successively refining the tree, obtaining tighter bounds after each refinement (see Corollary \ref{cor:lb}). The algorithm solves a version of the mean value lower bound problem \eqref{mod:mvp-part} for successively more refined trees $\parttree$. Here we present a modified version of \eqref{mod:mvp-part} that better facilitates the refining of the partition tree in a dynamic fashion. First, we expand the expression $\mathbb{P}[\leaf\,|\,\firststageallocvec]$ in \eqref{mod:mvp-part} as the product of the conditional probabilities of the nodes in the sequence of nodes along the path from the root node to $\leaf$. Let $\treenodesonpath(\leaf)$ denote the set of nodes along the path from the root node to leaf $\leaf$, $\treenode_{\pathindex,\lambda}$ the $\pathindex$th node along the path, and $\mathbb{P}[\treenode_{\pathindex+1,\lambda}\,|\,\firststageallocvec,\treenode_{i,\lambda}]$ the conditional probability of reaching node $\treenode_{\pathindex+1,\lambda}$ from node $\treenode_{\pathindex,\lambda}$. (Note that this is simply the probability of the realization the capacity level represented by the arc from $\treenode_{\pathindex,\lambda}$ to $\treenode_{\pathindex+1,\lambda}$, i.e, $\alloclevelssum \stateprobfn{i}{\componoutcome_{\compon\treenode_{\pathindex+1,\leaf}};\alloclevel} \firststageallocil$, where $\compon$ is the component that became fixed at tree node $\treenode_{\pathindex+1,\leaf}$.) Then the expanded formulation is:

\begin{mini}|s|
    {\firststagevec\in \firststagefeasregion}{\costallocvec^T\firststagevec + \sum_{\leaf\in \parttreeleafnodes(\parttree)} \left(\prod_{i=1}^{|\treenodesonpath(\lambda)|} \mathbb{P}[\treenode_{i+1,\lambda}\,|\,\firststageallocvec,\treenode_{i,\lambda}]\right) \recfn(\bar{\staterealizvar}_{\leaf}(\firststageallocvec))}{\label{mod:mvp-part-expand}}{}
\end{mini}

Next, we reformulate this problem via an epigraph formulation so that refining the tree can be done dynamically by adding constraints to the mean value lower bound problem. This epigraph formulation allows the successive refinement algorithm to be implemented in within a branch-and-cut framework by dynamically adding rows when a new incumbent solution is found. In this reformulation, we first use an auxiliary variable $\theta_0$ to represent the objective value $\ubar{w}(\parttree)$. We also use auxiliary variables $\auxvar_\treenode$ to represent the objective value at node $\treenode\in\parttreenodes(\parttree)$, conditional upon the capacity levels of the fixed components at that node. The model also includes copies of the second stage variables for each leaf node, i.e., $\secondstagevec^\leaf$. Let $\descendentnodes(\treenode;\parttree)$ denote the set of descendent nodes of node $\treenode$ in partition tree $\parttree$. Then the reformulation is as follows:

\begin{mini!}|s|
    {\firststageallocvec\in \firststagefeasregion}{\auxvar_0 \label{eq:epi-obj}}{\label{mod:epi}}{\ubar{w}(\parttree)=}
    \addConstraint{\auxvar_\treenode}{\geq \sum_{\treenode'\in \descendentnodes(\treenode;\parttree)} \mathbb{P}[\treenode'\,|\,\firststageallocvec,\treenode] \auxvar_{\treenode'}}{\quad \forall \treenode\in \parttreenodes(\parttree)\setminus \parttreeleafnodes(\parttree)\label{eqn:conditional}}
    \addConstraint{\auxvar_\leaf}{\geq \secondstagecostvec^T \secondstagevec^\leaf}{\quad \forall \leaf\in \parttreeleafnodes(\parttree) \label{eqn:leaves}}
    \addConstraint{\secondstagematrix\secondstagevec^\leaf}{= \secondstagerhsvec}{\quad \forall \leaf\in \parttreeleafnodes(\parttree) \label{eqn:leaves2}}
    \addConstraint{\secondstagevec^\leaf}{\leq u\, \bar{\staterealizvar}_{\leaf}(\firststageallocvec)}{\quad \forall \leaf\in \parttreeleafnodes(\parttree). \label{eqn:leaf-cap}}
\end{mini!}

Constraints \eqref{eqn:conditional} ensure that the conditional objective value at a node $\treenode$ is equal to the expected objective value over its descendants, conditional on the realized capacity levels of the fixed components at node $\treenode$, $\fixedcomponents(\treenode)$. Note that the conditional probability $\mathbb{P}[\treenode'\,|\,\firststageallocvec,\treenode]$ is the probability of the realization of the component corresponding to the partition tree arc in between nodes $\treenode$ and $\treenode'$, i.e., the component that became fixed at node $\treenode'$. In other words, $\mathbb{P}[\treenode'\,|\,\firststageallocvec,\treenode] = \alloclevelssum \stateprobfn{i}{\componoutcome_{\compon\treenode'};\alloclevel} \firststageallocil$, where $\compon$ is the component that became fixed at node $\treenode'$.

Constraints \eqref{eqn:leaves} compute the second-stage objective value at each leaf $\leaf$, conditional on the capacity levels of the fixed components at $\leaf$. The constraint is an inequality rather than an equality so that this constraint does not have to be removed if the partition tree to be refined by spawning descendent of $\leaf$. If such a refinement does occur, new constraints \eqref{eqn:conditional} can be added to the model without having to remove any constraints (see \S\ref{sec:update-model}). Constraints \eqref{eqn:leaves2} and \eqref{eqn:leaf-cap} enforce the second-stage constraints at leaf node $\leaf$ so that $\auxvar_\leaf=\recfn(\bar{\staterealizvar}_{\leaf}(\firststageallocvec))$ for all $\leaf\in\parttreeleafnodes(\parttree)$.

\subsection{Selecting A Leaf to Subdivide}
After solving \eqref{mod:epi} for a particular partition tree $\parttree$ and obtaining a solution $\firststageallocvec$, the next step is to refine the tree. In our approach we refine the tree at a single leaf node $\leaf'\in \parttreeleafnodes(\parttree)$. To choose the leaf node to refine, we first observe that for a given solution $\firststageallocvec$ the gap between the upper bound and lower bounds of the true objective value for that solution is (see Theorem \ref{thm:lb} and Corollary \ref{cor:ub}):

\[\bar{w}(\firststageallocvec) - \ubar{w}(\firststageallocvec) = \sum_{\leaf\in \parttreeleafnodes(\parttree)} \mathbb{P}[\leaf\,|\,\firststageallocvec] \left[\tilde{\recfn}(\bar{\staterealizvar}_{\leaf}(\firststageallocvec)) - \recfn(\bar{\staterealizvar}_{\leaf}(\firststageallocvec))\right]\]

Thus, the probability-weighted difference for leaf node $\leaf$ is

\[D^\leaf(\firststageallocvec) = \mathbb{P}[\leaf\,|\,\firststageallocvec](\tilde{\recfn}(\bar{\staterealizvar}_{\leaf}(\firststageallocvec)) - \recfn(\bar{\staterealizvar}_{\leaf}(\firststageallocvec))),\]

\noindent where $\mathbb{P}[\leaf\,|\,\firststageallocvec]$ is the probability of leaf $\leaf$ given fixed solution $\firststageallocvec$. We select the leaf that maximizes this weighted difference, i.e., $\leaf'\in \arg\max_{\leaf\in \parttreeleafnodes(\parttree)} D^\leaf(\firststageallocvec)$.

\subsection{Subdividing on a Component}
Once we have selected $\leaf'$ as the leaf node to subdivide, we then select a component to subdivide on. Specifically, we examine all unfixed components and select the one that maximally reduces the resulting difference between the upper and lower bounds for current solution $\firststageallocvec$. Let $\Lambda(\lambda,i,\componoutcome)$ be the descendent of leaf $\leaf$ that results from setting component $\compon$ to capacity level $\componoutcome$. Then we choose a component $\compon'$ to subdivide on such that:

\[\compon'\in \arg\min_{\compon=1,\dots,\numcompon}\left\{ \outcomessum \componoutcome \left(\alloclevelssum \stateprobfn{i}{\componoutcome;\alloclevel} \firststageallocil\right) D^{\Lambda(\leaf,i,\componoutcome)}(\firststageallocvec)\right\}, \]

\noindent where the summation in the inner parentheses denotes the probability that component $\compon$ has a state of $\componoutcome$. 

\subsection{Refining the Tree and Updating the Mean Value Lower Bound Problem}\label{sec:update-model}

Given an existing tree $\parttree=(\parttreenodes,\parttreearcs)$, the function $\textsc{Refine}(\parttree,\leaf',i')$ produces a new tree $\parttree'$ with nodes set that does not include $\leaf'$ but does include the descendants of $\leaf'$, i.e., $\parttreenodes' = (\parttreenodes\setminus\{\leaf'\})\cup \bigcup_{\componoutcome=0}^\numoutcomes \Lambda(\leaf',i',\componoutcome)$. After a tree $\parttree$ is refined into the new tree $\parttree'$, the conditional mean value lower bound model \eqref{mod:epi} is refined as follows. Letting $\leaf'$ be the leaf that is to be that is to be refined, refining the model amounts to adding new constraint \eqref{eqn:conditional} for $\treenode=\leaf'$ and adding constraints \eqref{eqn:leaves}-\eqref{eqn:leaf-cap} for all $\leaf\in \descendentnodes(\leaf')$, where $\descendentnodes(\leaf') = \bigcup_{\componoutcome=0}^\numoutcomes \Lambda(\leaf',i',\componoutcome)$. Note that because of the convexity of $\recfn$, the new constraint \eqref{eqn:conditional} for $\treenode=\leaf'$ will be tighter than the old constraints \eqref{eqn:leaves}-\eqref{eqn:leaf-cap} for $\leaf'$. Thus, the model is correct even if the old constraints constraints \eqref{eqn:leaves}-\eqref{eqn:leaf-cap} are not removed.

\subsection{Successive Refinement Algorithm}
Algorithm \ref{sra-callback} contains pseudocode for the successive refinement algorithm.

\begin{algorithm}
	{\sc SuccessiveRefinementAlgorithm}($\epsilon$)
	\begin{algorithmic}[1]
        %\STATE Choose some small number $\epsilon > 0$.
        \STATE $\mathcal{T}\gets\mathcal{T}_0$
        \STATE $UB\gets +\infty$. 
        \STATE $Q \gets \emptyset$
        \STATE \textbf{ADD} initial mean value lower bound problem \eqref{mod:epi} to queue of active nodes $Q$.
		\WHILE{$|Q| >0 $}
        \STATE Use \textit{node selection procedure} to dequeue node $\bbnode$ from $Q$.
        \STATE Solve relaxation of node $\bbnode$ problem. %$LB\gets z^*$.
        \IF{relaxed problem is feasible}
        \STATE Let $\firststageallocvec$ denote the solution and $z^*$ the objective value.
        \IF{$\firststageallocvec$ is integral}
        \STATE $UB\gets\{UB,z^*\}$.
        \IF{$\firststageallocvec$ is a new incumbent.}
        \STATE let $\leaf'\in \arg\max_{\leaf\in \mathcal{L}(\mathcal{T})} \{D^\leaf(\firststageallocvec)\}$.
        \IF{$D^{\leaf'}(\firststageallocvec) > \epsilon$}
        \STATE let $\compon'\in \arg\min_{\compon=1,\dots,\numcompon}\left\{ \outcomessum \left(\alloclevelssum \stateprobfn{i}{\componoutcome;\alloclevel} \firststageallocil\right) D^{\Lambda(\leaf',\compon,\componoutcome)}(\firststageallocvec)\right\}$.
        \STATE $\mathcal{T}\gets \textsc{Refine}(\mathcal{T},\leaf',\compon')$
        \STATE Update mean value lower bound problem \eqref{mod:epi} based on new tree $\mathcal{T}$ (see \S\ref{sec:update-model}).
        \ENDIF
        \ENDIF
        \ELSE
        \STATE Execute \textit{branching procedure} to create children of $\bbnode$ and add them to $Q$. 
        \ENDIF
        \ENDIF
		\ENDWHILE
		\RETURN $UB$, $x^*$
	\end{algorithmic}
	\caption{Successive refinement algorithm (within basic branch-and-bound procedure).}
	\label{sra-callback}
\end{algorithm}

\section{Experimental Results}
\label{expr-results}

% \subsection{Test Problems}\label{test-problems}

\subsection{Test Problem: Stochastic Network Interdiction Problem}\label{snip}

We tested our approach using a stochastic network interdiction problem. The first-stage decision variables denote an interdictor's allocation of attack resources to arcs, while the recourse objective is to maximize the expected maximum flow of the network across possible failure scenarios. Consider a graph $G = (N, A)$ where $N$ is the set of nodes and $A$ is the set of arcs. Let $\capac[k]$ denote the capacity of arc $k$ (in our formulation arcs are denoted by a single index $k$; we use the index $i$ to denote nodes in the set $N$). Let  $\staterealizvar$ be a random vector such that $\staterealizvar_k = 1$ if arc $k$ is interdicted and 0 otherwise.  Thus, $\componoutcome=1$ denotes that an arc has failed and $\kappa=0$ denotes that it is available. Let $b^A$ be the attacker's budget and let $\firststagefeasregion = \{\firststageallocvec \, |\, \sum_{\compon \in A} \alloclevelssum \alloclevel\firststagealloc[k\alloclevel] \leq b^A\,\numalloclevels\}$ the set of feasible attacker decisions. The objective function of the interdiction problem is given as:

\begin{mini}|s|
    {\firststageallocvec\in \firststageallocfeasregion}{ \sum_{\staterealizvar\in \randvecsupport} \mathbb{P}[\staterealizvar \,|\, \firststageallocvec]\,\recfn(\staterealizvar)}{\label{mod:innterdiction first-stage}}{}
\end{mini}

The recourse function $\recfn(\staterealizvar)$ represents the maximum flow given arc availability vector $\staterealizvar$:

\begin{maxi!}|s|
    {}{\secondstagevar_{ts} \label{eq:max-flow-obj}}{\label{mod:max-flow-recourse}}{\recfn(\staterealizvar) =}
    \addConstraint{\sum_{k\in A^+_i} \secondstagevar_k - \sum_{k\in A^-_i} \secondstagevar_k}{= \delta(i)}{\quad \forall i\in N\label{flow-conserve}}
    \addConstraint{0 \leq \secondstagevar_\compon}{\leq \capac[k](1-\staterealizvar_k)}{\quad \forall k\in A, \label{max-flow-cap}}
\end{maxi!}

\noindent where $s$ is the source node, $t$ is the sink node, $\secondstagevar_{ts}$ represents the flow on reverse arc $(t,s)$, $A_i^+$ and $A_i^-$ are the forward and reverse star sets for node $i$ and $\delta(i)$ is an expression that equals $-y_{ts}$ for $i=s$ and $y_{ts}$ for $i=t$ and 0 for all other $i\in N\setminus\{s,t\}$.

In order to solve this min-max formulation, we use the dualize-and-combine approach in which we take the dual of the inner maximum flow problem and then combine it with the first-stage minimization problem. The dual model is:

\begin{mini!}|s|
    {}{\sum_{k\in A} \capac[k](1-\staterealizvar_k) \capdual[k] \label{eq:max-flow-dual-obj}}{\label{mod:max-flow-recourse-dual}}{}
    \addConstraint{\alpha_i - \alpha_j + \beta_{ij}}{\geq 0}{\quad \forall (i,j)=k\in A\label{reduced-cost}}
    \addConstraint{0\leq \beta_k}{\leq 1}{\quad \forall k\in A\label{beta-bounds}}
    \addConstraint{\alpha_s}{=0}{\label{alphaS}}
    \addConstraint{\alpha_t}{=1.}{\label{alphaT}}
\end{mini!}

Thus, the mean value lower bound problem is:

\begin{mini!}|s|
    {\firststageallocvec\in \firststagefeasregion}{\auxvar_0 \label{eq:epi-interdict-obj}}{\label{mod:interdict-epi}}{\ubar{w}(\parttree)=}
    \addConstraint{\auxvar_\treenode}{\geq \sum_{\treenode'\in \descendentnodes(\treenode;\parttree)} \mathbb{P}[\treenode'\,|\,\firststageallocvec,\treenode] \auxvar_{\treenode'}}{\quad \forall \treenode\in \parttreenodes(\parttree)\setminus \parttreeleafnodes(\parttree)\label{eqn:interdict-conditional}}
    \addConstraint{\auxvar_\leaf}{\geq \sum_{k\in A} \capac[k](1-\bar{\staterealizvar}_{\leaf}(\firststageallocvec)) \capdual[k]^\leaf}{\quad \forall \leaf\in \parttreeleafnodes(\parttree) \label{eqn:interdict-leaves}}
    \addConstraint{\alpha_i^\leaf - \alpha_j^\leaf + \beta_{ij}^\leaf}{\geq 0}{\quad \forall (i,j)=k\in A, \leaf\in \parttreeleafnodes(\parttree) \label{eqn:interdict-leaves2}}
    \addConstraint{0\leq \beta_k^\leaf}{\leq 1}{\quad \forall (i,j)=k\in A, \leaf\in \parttreeleafnodes(\parttree) \label{eqn:interdict-leaf-cap}}
    \addConstraint{\alpha_s^\leaf}{=0}{\quad \forall\leaf\in \parttreeleafnodes(\parttree)\label{alphaS-mvlb}}
    \addConstraint{\alpha_t^\leaf}{=1}{\forall\leaf\in \parttreeleafnodes(\parttree).\label{alphaT-mvlb}}
\end{mini!}

And the penalty formulation used to compute upper bounds is:

\begin{maxi!}|s|
    {}{y_{ts} - \sum_{k\in A } (1-\staterealizvar_k)\secondstagevar_k \label{eq:max-flow-obj-penalty}}{\label{mod:max-flow-penalty-recourse}}{\recfnpenalty(\staterealizvar) =}
    \addConstraint{\sum_{k\in A^+_i} \secondstagevar_k - \sum_{k\in A^-_i} \secondstagevar_k}{= \delta(i)}{\quad \forall i\in N\label{max-flow-penalty-flow-conserve}}
    \addConstraint{0 \leq \secondstagevar_k}{\leq \capac[k]}{\quad \forall k\in A. \label{max-flow-penalty-cap}}
\end{maxi!}

In this test problem, we assumed that arcs emanating from the source or flowing into the sink cannot fail (to avoid easily disconnecting the network) and set the capacity of these arcs (as well as the reverse arc $(t,s)$) to a large value. We used the following component state probability function

\begin{equation}\label{state-prob-fn}
    f_k(\componoutcome;\alloclevel) = \begin{cases}
\frac{\alloclevel}{\alloclevel + a_k}, & \componoutcome = 1,\alloclevel > 0\\
1- \frac{\alloclevel}{\alloclevel + a_k}, & \componoutcome = 0,\alloclevel > 0\\
0, & \componoutcome = 1,\alloclevel = 0,a_k = 0\\
1, & \componoutcome = 0,\alloclevel = 0,a_k = 0,
\end{cases}
\end{equation}

\noindent where $a_k$ is a parameter varied in experiments. Thus, this state probability function implies that if $a_k = 0$ then attacking an arc with any amount of resources guarantees that the arc will fail. Conversely, if zero attack resources are allocated to an arc, then it is guaranteed to not fail. Also, arcs with $a_k=0$ that are not attacked never fail. Thus, it was assumed that arcs $a_k =0$ cannot fail; hence, these arcs have fixed, not uncertain, capacity. Thus, increasing the number of arcs with $a_k>0$ increases the cardinality of the support $\randvecsupport$, i.e., the number of scenarios.

\subsection{Benchmark Method}\label{benchmark}
To assess the approach's effectiveness, we use computational results from solving the multilinear stochastic programming (MSP) formulation \eqref{mod:multi-program} using Gurobi. We used the standard approach of sequentially decomposing the nonlinear terms into multiple bilinear functions to address the nonlinearity issue in \eqref{mod:multi-program}. Then, we apply the McCormick envelopes to these terms. For every bilinear term, we introduce an auxiliary variable, and along with each auxiliary variable, we add four constraints based on McCormick's technique. The size of the resulting model is significantly larger due to the number of auxiliary variables, constraints added, and the fact that there are distinct copies of $\secondstagevec$ for each scenario. 

\subsection{Results}\label{results}
We applied the SRA to solve the stochastic network interdiction problem, described in  \S\ref{snip} on varying grid networks with equal numbers of rows and columns. The network grids are connected as shown in Fig \ref{fig:3} and have bi-directional arcs with the source nodes connected to all external nodes on the left side of the grid, and the target is connected to all the nodes on the rightmost side. We consider $3\times3, 4\times4,5\times5,6\times6,7\times7, 8\times8$  grid networks with $30, 56, 90, 132, 182, 240$ arcs, respectively, while varying the interdiction budget ($b^A$) and the number of allocation levels ($\numalloclevels$) for each network size. We implemented SRA in Gurobi using callback functions while using Gurobi to solve the second-stage problems. Since Gurobi doesn't allow the dynamic addition of variables during the branch-and-cut algorithm, we added all the required variables at the start but only used them when a new constraint with the variable was added to the model. All computational experiments were conducted on a server allocated with a single dedicated thread. We set the time limit to $1800s$ for experiments with network sized $3\times3$, $3600s$ for network sized $4\times4$, and $7200s$ for all other network sizes. 

% For comparison, we computed results using the simpler multilinear stochastic program \eqref{mod:multi-program2} and solved two deterministic network protection problems to determine which components belong to the subset $J$. In these models, each network component is assigned its expected capacity. In the first model, components can receive varying protection levels, mirroring the interdiction levels. This setup enables the distribution of a limited protection budget across as many components as possible by allocating the minimal protection level to each. Conversely, the second case allows only the highest level of protection for each component, resulting in fewer protected components. The components selected for optimal protection in either scenario are included in subset $J$. 

\begin{figure}
\centering
  \includegraphics[width=0.75\linewidth]{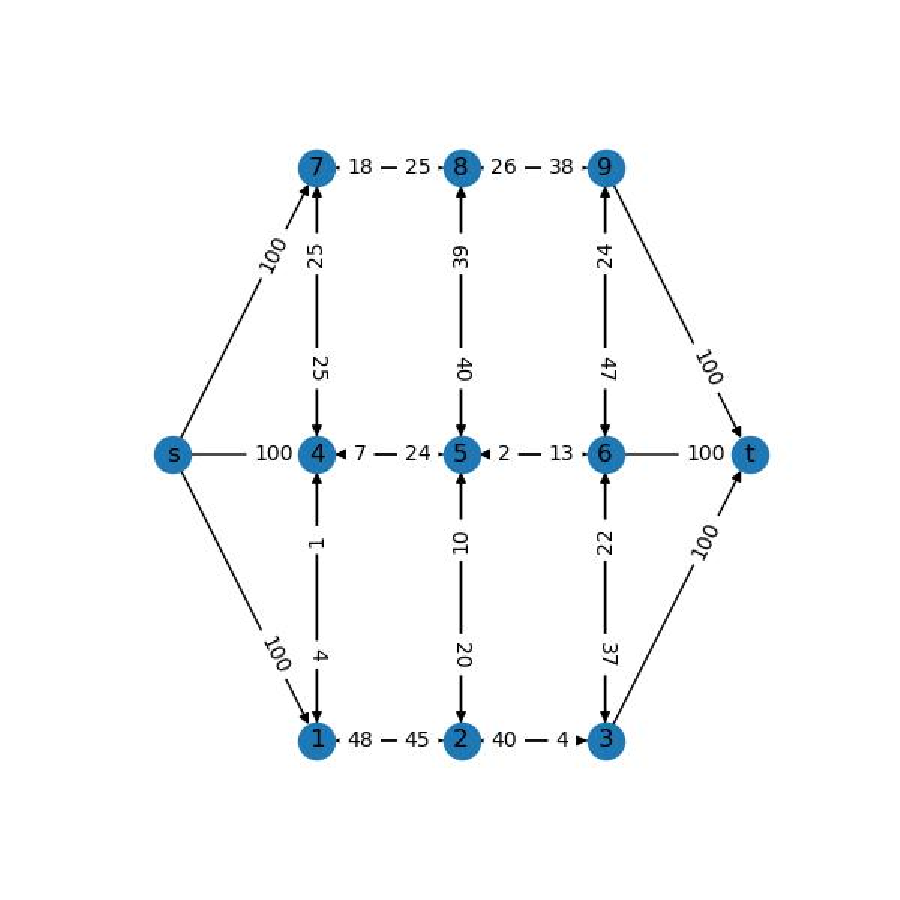}
\caption{A grid network with 3 rows and 3 columns}
\label{fig:3}
\end{figure}

Our study categorizes the experimental runs based on the way in which the parameter $a_k$ in \eqref{state-prob-fn} is computed:
\begin{itemize}
    \item \textbf{Det-Opt-Bin}: A deterministic tri-level defender-attacker-defender fortification model with binary fortification variables and a defender budget of $b^A$ was solved and the $a_k=\numalloclevels$ for arcs that are hardened in the defenders solution. This approach results in a small number of arcs with non-zero $a_k$.
    \item \textbf{Det-Opt}: A deterministic tri-level defender-attacker-defender fortification model with binary fortification variables and defender budget of $b^A\,\numalloclevels$ was solved and the $a_k=1$ for arcs that are hardened in the defenders solution. This approach results in more arcs with non-zero $a_k$.
    \item \textbf{Uniform}: $a_k = (b^A\,\numalloclevels)/|A'|$, where $A'$ is the subset of arcs in $A$ that can fail. This approach results in a all arcs having $a_k>0$, making the problem difficult to solve.
\end{itemize}

% Experiments that assume probabilistic failure across all components are denoted as "uniform." Meanwhile, experiments that apply the subset $J$ for component failures fall into two categories: "Det-Opt" for the scenario allowing variable levels of protection and "Det-Opt-Bin" for the scenario with protection limited to the maximum allocation level for each component. 

We average the results for each experiment over 5 randomly generated instances for each network size, budget ($b^A$), and number of allocation levels ($\numalloclevels$) combination. We note the model's gap when the runtime exceeds the set time limit. If an experiment finishes within this limit, we instead report its runtime. Table \ref{table: det-opt-bin-res} presents outcomes for both MSP and SRA approaches applied to the \textbf{Det-Opt-Bin} cases. Similarly, results for the \textbf{Det-Opt} model, under the same approaches, are detailed in Table \ref{table:det-opt-res}. For the \textbf{Uniform} case, only the SRA method's results appear in Table \ref{table: unif-res}, as the MSP model didn't complete these problems due to the fact that the large number of arcs with $a_k>0$ results in a large number of scenarios. Table \ref{table: det-opt-bin-res} reveals that the MSP method effectively solves the \textbf{Det-Opt-Bin} case for networks up to size 8, with budgets and protection levels from 2 to 4, within acceptable runtimes due to the fact that less arcs have $a_k>0$. SRA significantly outperforms MSP in speed and scalability for the same model. For the \textbf{Det-Opt} model, MSP's effectiveness is limited to smaller problems with a budget of 2, attributed to a larger number of arcs with $a_k>0$ and thus higher non-linearity. SRA, however, manages to solve or nearly solve (with gaps below $1\%$) most \textbf{Det-Opt} instances within time constraints. This efficiency extends to the \textbf{Uniform} cases, as shown in Table \ref{table: unif-res}, where SRA consistently achieves solutions or near-optimal results under similar conditions.

The results consistently indicate that the SRA algorithm needs more time to compute solutions as we expand the network in size or increase either the budget or the allocation levels. This is because larger networks have more arcs, adding complexity and nonlinearity to the problem. Such complexity deepens the refinement tree necessary for SRA, requiring more extensive refinement efforts to close the optimality gap. Higher budgets exacerbate this effect, diverging further from the mean value problem's optimal objective (referenced in \eqref{eq:mvp}), necessitating additional refinements to narrow the gap.

\begin{table}[htbp]
\centering
\caption{DET-OPT-BIN RESULTS}
\label{table: det-opt-bin-res}
    \begin{tabular}{|ccc|cc|ccc|}
        \hline
        \multicolumn{3}{|l|}{} &
          \multicolumn{2}{c|}{\textbf{MSP}} &
          \multicolumn{3}{c|}{\textbf{SRA}} \\ \hline
        \multicolumn{1}{|l|}{\textbf{rows $\times$ columns}} &
          \multicolumn{1}{l|}{$\textbf{b}^\textbf{A}$} &
          $\textbf{\numalloclevels}$ &
          \multicolumn{1}{c|}{\textbf{Runtime}} &
          \textbf{Solved} &
          \multicolumn{1}{c|}{\textbf{Runtime}} &
          \multicolumn{1}{c|}{\textbf{No of Refinements}} &
          \textbf{Solved} \\ \hline
        $3\times3$ & 2 & 2 & 0.1   & 5 & 0.3 & 0.8 & 5 \\
        $3\times3$ & 2 & 3 & 0.3   & 5 & 0.4 & 0.8 & 5 \\
        $3\times3$ & 2 & 4 & 1   & 5 & 0.3 & 0.8 & 5 \\
        $3\times3$ & 3 & 2 & 0.5   & 5 & 0.3 & 1.2 & 5 \\
        $3\times3$ & 3 & 3 & 0.4   & 5 & 0.3 & 1.6 & 5 \\
        $3\times3$ & 3 & 4 & 0.7   & 5 & 0.3 & 1.2 & 5 \\
        $3\times3$ & 4 & 2 & 3   & 5 & 0.5 & 3.6 & 5 \\
        $3\times3$ & 4 & 3 & 2   & 5 & 0.7 & 6.0 & 5 \\
        $3\times3$ & 4 & 4 & 3   & 5 & 0.8 & 5.0 & 5 \\ \hline
        $4\times4$ & 2 & 2 & 0.2   & 5 & 0.3 & 0.0 & 5 \\
        $4\times4$ & 2 & 3 & 0.9   & 5 & 0.3 & 0.0 & 5 \\
        $4\times4$ & 2 & 4 & 3   & 5 & 0.3 & 0.0 & 5 \\
        $4\times4$ & 3 & 2 & 1   & 5 & 0.4 & 0.4 & 5 \\
        $4\times4$ & 3 & 3 & 0.7   & 5 & 0.5 & 0.6 & 5 \\
        $4\times4$ & 3 & 4 & 2   & 5 & 0.7 & 0.6 & 5 \\
        $4\times4$ & 4 & 2 & 7   & 5 & 0.9 & 2.6 & 5 \\
        $4\times4$ & 4 & 3 & 5   & 5 & 1 & 3.8 & 5 \\
        $4\times4$ & 4 & 4 & 7   & 5 & 2 & 3.0 & 5 \\ \hline
        $5\times5$ & 2 & 2 & 0.4   & 5 & 0.5 & 0.0 & 5 \\
        $5\times5$ & 2 & 3 & 2   & 5 & 0.4 & 0.0 & 5 \\
        $5\times5$ & 2 & 4 & 8   & 5 & 0.4 & 0.0 & 5 \\
        $5\times5$ & 3 & 2 & 2   & 5 & 0.4 & 0.0 & 5 \\
        $5\times5$ & 3 & 3 & 1   & 5 & 0.4 & 0.0 & 5 \\
        $5\times5$ & 3 & 4 & 3   & 5 & 0.4 & 0.0 & 5 \\
        $5\times5$ & 4 & 2 & 19  & 5 & 1 & 1.8 & 5 \\
        $5\times5$ & 4 & 3 & 11  & 5 & 2 & 1.6 & 5 \\
        $5\times5$ & 4 & 4 & 37  & 5 & 4 & 1.4 & 5 \\ \hline
        $6\times6$ & 2 & 2 & 0.6   & 5 & 0.6 & 0.0 & 5 \\
        $6\times6$ & 2 & 3 & 3   & 5 & 0.6 & 0.0 & 5 \\
        $6\times6$ & 2 & 4 & 17  & 5 & 0.6 & 0.0 & 5 \\
        $6\times6$ & 3 & 2 & 3   & 5 & 0.6 & 0.0 & 5 \\
        $6\times6$ & 3 & 3 & 2   & 5 & 0.6 & 0.0 & 5 \\
        $6\times6$ & 3 & 4 & 6   & 5 & 0.6 & 0.0 & 5 \\
        $6\times6$ & 4 & 2 & 27  & 5 & 0.6 & 0.0 & 5 \\
        $6\times6$ & 4 & 3 & 201  & 5 & 0.7 & 0.2 & 5 \\
        $6\times6$ & 4 & 4 & 81  & 5 & 0.6 & 0.0 & 5 \\ \hline
        $7\times7$ & 2 & 2 & 0.7   & 5 & 0.8 & 0.0 & 5 \\
        $7\times7$ & 2 & 3 & 5   & 5 & 0.8 & 0.0 & 5 \\
        $7\times7$ & 2 & 4 & 23  & 5 & 0.8 & 0.0 & 5 \\
        $7\times7$ & 3 & 2 & 7   & 5 & 0.8 & 0.0 & 5 \\
        $7\times7$ & 3 & 3 & 3   & 5 & 0.8 & 0.0 & 5 \\
        $7\times7$ & 3 & 4 & 9   & 5 & 0.8 & 0.0 & 5 \\
        $7\times7$ & 4 & 2 & 41  & 5 & 0.8 & 0.0 & 5 \\
        $7\times7$ & 4 & 3 & 33  & 5 & 0.8 & 0.0 & 5 \\
        $7\times7$ & 4 & 4 & 168 & 5 & 0.8 & 0.0 & 5 \\ \hline
        $8\times8$ & 2 & 2 & 1   & 5 & 1 & 0.0 & 5 \\
        $8\times8$ & 2 & 3 & 9   & 5 & 1 & 0.0 & 5 \\
        $8\times8$ & 2 & 4 & 31  & 5 & 1 & 0.0 & 5 \\
        $8\times8$ & 3 & 2 & 8   & 5 & 1 & 0.0 & 5 \\
        $8\times8$ & 3 & 3 & 5   & 5 & 1 & 0.0 & 5 \\
        $8\times8$ & 3 & 4 & 14  & 5 & 1 & 0.0 & 5 \\
        $8\times8$ & 4 & 2 & 44  & 5 & 1 & 0.0 & 5 \\
        $8\times8$ & 4 & 3 & 49  & 5 & 1 & 0.0 & 5 \\
        $8\times8$ & 4 & 4 & 283 & 5 & 1 & 0.5 & 2 \\ \hline
    \end{tabular}
\end{table}

\begin{table}[htbp]
\centering
\caption{DET-OPT RESULTS}
\label{table:det-opt-res}
\begin{tabular}{|ccc|cc|cccc|}
    \hline
    \multicolumn{3}{|l|}{}               & \multicolumn{2}{c|}{\textbf{MSP}} & \multicolumn{4}{c|}{\textbf{REFINE}}                                           \\ \hline
    \textbf{rows $\times$ columns} & $\textbf{b}^\textbf{A}$ & $\textbf{\numalloclevels}$ & \textbf{Runtime}  & \textbf{Solved} & \textbf{Gap} & \textbf{Runtime} & \textbf{No of Refinements} & \textbf{Solved} \\ \hline
    $3\times3$ & 2 & 2 & 2   & 5 &      & 0.6    & 4.2   & 5 \\
    $3\times3$ & 2 & 3 & 41  & 5 &      & 1    & 9.8   & 5 \\
    $3\times3$ & 2 & 4 &        & 0 &      & 4    & 15.2  & 5 \\
    $3\times3$ & 3 & 2 &        & 0 &      & 1    & 10.2  & 5 \\
    $3\times3$ & 3 & 3 &        & 0 &      & 4    & 16.4  & 5 \\
    $3\times3$ & 3 & 4 &        & 0 &      & 19   & 22.6  & 5 \\
    $3\times3$ & 4 & 2 &        & 0 &      & 3    & 15.8 & 5 \\
    $3\times3$ & 4 & 3 &        & 0 &      & 6    & 22.0  & 5 \\
    $3\times3$ & 4 & 4 &        & 0 &      & 26   & 30.8  & 5 \\ \hline
    $4\times4$ & 2 & 2 & 7   & 5 &      & 0.5    & 1.2   & 5 \\
    $4\times4$ & 2 & 3 & 71  & 5 &      & 2    & 5.0   & 5 \\
    $4\times4$ & 2 & 4 &        & 0 &      & 14   & 16.2  & 5 \\
    $4\times4$ & 3 & 2 &        & 0 &      & 1    & 5.6   & 5 \\
    $4\times4$ & 3 & 3 &        & 0 &      & 65   & 31.4  & 5 \\
    $4\times4$ & 3 & 4 &        & 0 &      & 818  & 69.8  & 5 \\
    $4\times4$ & 4 & 2 &        & 0 &      & 8    & 22.8  & 5 \\
    $4\times4$ & 4 & 3 &        & 0 &      & 793  & 96.8  & 5 \\
    $4\times4$ & 4 & 4 &        & 0 & 0.82 & 2118 & 195.8 & 1 \\ \hline
    $5\times5$ & 2 & 2 & 12  & 5 &      & 0.7    & 0.8   & 5 \\
    $5\times5$ & 2 & 3 & 146 & 5 &      & 2    & 2.8   & 5 \\
    $5\times5$ & 2 & 4 &        & 0 &      & 24   & 11.2  & 5 \\
    $5\times5$ & 3 & 2 &        & 0 &      & 2    & 3.2   & 5 \\
    $5\times5$ & 3 & 3 &        & 0 &      & 79   & 16.4  & 5 \\
    $5\times5$ & 3 & 4 &        & 0 & 0.93 & 2100 & 45.4  & 4 \\
    $5\times5$ & 4 & 2 &        & 0 &      & 12   & 14.4  & 5 \\
    $5\times5$ & 4 & 3 &        & 0 &      & 2009 & 58.2  & 5 \\
    $5\times5$ & 4 & 4 &        & 0 & 1.00 & 0.4    & 51.0  & 4 \\ \hline
    $6\times6$ & 2 & 2 & 23  & 5 &      & 0.8    & 0.2   & 5 \\
    $6\times6$ & 2 & 3 &        & 0 &      & 2    & 1.6   & 5 \\
    $6\times6$ & 2 & 4 &        & 0 &      & 13   & 4.6   & 5 \\
    $6\times6$ & 3 & 2 &        & 0 &      & 2    & 2.0   & 5 \\
    $6\times6$ & 3 & 3 &        & 0 &      & 43   & 7.8   & 5 \\
    $6\times6$ & 3 & 4 &        & 0 & 0.78 & 2045 & 22.4  & 3 \\
    $6\times6$ & 4 & 2 &        & 0 &      & 8    & 5.2   & 5 \\
    $6\times6$ & 4 & 3 &        & 0 & 0.75 & 690  & 27.4  & 3 \\
    $6\times6$ & 4 & 4 &        & 0 & 1 &         & 62.4  & 0 \\ \hline
    $7\times7$ & 2 & 2 & 34  & 5 &      & 1    & 0.2   & 5 \\
    $7\times7$ & 2 & 3 &        & 0 &      & 2    & 0.8   & 5 \\
    $7\times7$ & 2 & 4 &        & 0 &      & 18   & 4.0   & 5 \\
    $7\times7$ & 3 & 2 &        & 0 &      & 2    & 1.0   & 5 \\
    $7\times7$ & 3 & 3 &        & 0 &      & 53   & 6.0   & 5 \\
    $7\times7$ & 3 & 4 &        & 0 & 0.81 & 2180 & 19.2  & 3 \\
    $7\times7$ & 4 & 2 &        & 0 &      & 14   & 5.2   & 5 \\
    $7\times7$ & 4 & 3 &        & 0 &      & 2872 & 18.8  & 5 \\
    $7\times7$ & 4 & 4 &        & 0 & 1 &         & 60.0  & 0 \\ \hline
    $8\times8$ & 2 & 2 &        & 0 &      & 1    & 0.0   & 5 \\
    $8\times8$ & 2 & 3 &        & 0 &      & 3    & 1.0   & 3 \\
    \hline
\end{tabular}
\end{table}

\begin{table}[htbp]
\centering
\caption{UNIFORM RESULTS FOR REFIMENT APPROACH}
\label{table: unif-res}
\begin{tabular}{|ccccccc|}
\hline
\textbf{rows $\times$ columns} & $\textbf{b}^\textbf{A}$ & $\textbf{\numalloclevels}$ & \textbf{Gap}   & \textbf{Runtime} & \textbf{No of Refinements} & \textbf{Solved} \\ \hline
$3\times3$ & 2 & 2 &       & 1    & 8.8              & 5      \\
$3\times3$ & 2 & 3 &       & 3    & 18.6             & 5      \\
$3\times3$ & 2 & 4 &       & 5    & 23.6             & 5      \\
$3\times3$ & 3 & 2 &       & 4    & 24.0             & 5      \\
$3\times3$ & 3 & 3 &       & 7    & 30.0             & 5      \\
$3\times3$ & 3 & 4 &       & 15   & 32.6             & 5      \\
$3\times3$ & 4 & 2 &       & 9    & 34.0             & 5      \\
$3\times3$ & 4 & 3 &       & 18   & 43.6             & 5      \\
$3\times3$ & 4 & 4 &       & 25   & 39.8             & 5      \\ \hline
$4\times4$ & 2 & 2 &       & 0.7    & 2.2              & 5      \\
$4\times4$ & 2 & 3 &       & 9    & 18.4             & 5      \\
$4\times4$ & 2 & 4 &       & 248  & 103.4            & 5      \\
$4\times4$ & 3 & 2 &       & 31   & 48.0             & 5      \\
$4\times4$ & 3 & 3 &       & 648  & 168.2            & 5      \\
$4\times4$ & 3 & 4 & 1 &         & 238.4            & 0      \\
$4\times4$ & 4 & 2 &       & 426  & 172.8            & 5      \\
$4\times4$ & 4 & 3 & 0.797 & 2404 & 241.2            & 2      \\
$4\times4$ & 4 & 4 & 1 &         & 255.4            & 0      \\ \hline
$5\times5$ & 2 & 2 &       & 2    & 3.8              & 5      \\
$5\times5$ & 2 & 3 &       & 30   & 22.2             & 5      \\
$5\times5$ & 2 & 4 &       & 1785 & 142.2            & 5      \\
$5\times5$ & 3 & 2 &       & 31   & 25.2             & 5      \\
$5\times5$ & 3 & 3 &       & 5548 & 216.8            & 5      \\
$5\times5$ & 3 & 4 & 1     & 0.5    & 165.6            & 1      \\
$5\times5$ & 4 & 2 &       & 1818 & 195.4            & 5      \\
$5\times5$ & 4 & 3 & 1     &         & 215.8            & 0      \\
$5\times5$ & 4 & 4 & 1      & 0    & 109.8              & 5      \\ \hline
$6\times6$ & 2 & 2 &       & 1    & 1.2              & 5      \\
$6\times6$ & 2 & 3 &       & 29   & 14.2             & 5      \\
$6\times6$ & 2 & 4 & 0.084 & 713  & 70.2             & 4      \\
$6\times6$ & 3 & 2 &       & 37   & 18.8             & 5      \\
$6\times6$ & 3 & 3 & 0.678 & 1719 & 120.0            & 1      \\
$6\times6$ & 3 & 4 & 1     & 1    & 62.2             & 2      \\
$6\times6$ & 4 & 2 &       & 2266 & 113.8            & 5      \\
$6\times6$ & 4 & 3 & 1     &         & 113.8            & 0  \\
$6\times6$ & 4 & 4 &       &         & 44.0             & 0      \\
\hline
$7\times7$ & 2 & 2 &       & 2    & 1.6              & 5      \\
$7\times7$ & 2 & 3 &       & 9    & 5.4              & 5      \\
$7\times7$ & 2 & 4 &       & 754  & 32.0             & 5      \\
$7\times7$ & 3 & 2 &       & 10    & 6.6              & 5      \\
$7\times7$ & 3 & 3 & 0.609 & 4381 & 105.6            & 1      \\
$7\times7$ & 3 & 4 & 1     &         & 79.0             & 0      \\
$7\times7$ & 4 & 2 &       & 1511 & 69.6             & 5      \\
$7\times7$ & 4 & 3 & 1     &         & 71.2             & 0      \\
$7\times7$ & 4 & 4 & 1     &         & 57.7             & 0      \\ \hline

\end{tabular}
\end{table}

%\textbf{Note:} the runtimes were rounded to the nearest tenth of a second for runtimes less than 1s and the nearest second for others. 

% \subsubsection{Reliable Capacitated Maximum Covering Location Problem}\label{rcflp-results}

% \section{Extensions}\label{extensions}

% \subsection{Continuous Random Variables}\label{cont-rv}

% \subsection{Continuous Decision Variables}\label{cont-dv}

\section{Conclusion}\label{conclusion}
This study presented a approach for solving a class of stochastic programming problems called \textit{stochastic programming with decision-dependent random capacities} (SP-DDRC). In this class of problems the second-stage problem includes a set of capacitated components, and the probability distribution of the capacity of each component depends on the first-stage variables. Thus, this class of problems is included in the class of stochastic programs with decision-dependent uncertainty. The SP-DDRC class of problems is very challenging to solve due to a large number of high-order multilinear terms. 

To alleviate the computational difficulty, we present a successive refinement algorithm (SRA) that is grounded in the fact that Jensen's inequality can be used to formulate a mean value problem that yields an elementary lower bound. The SRA starts with this elementary lower bound and refines the support of the random capacities dynamically within a branch-and-bound search.

We conducted a set of computational tests and found that the SRA effectively manages the computational complexities of endogenous probabilities in SP-DDRCs, demonstrating significant improvement in runtime and solution accuracy compared to the benchmark multilinear stochastic programming formulation. This is evident in our detailed analysis of network interdiction problems. SRA not only scaled with increasing network size and complexity but also obtained exact solutions in most cases under diverse budget and allocation scenarios. 

While this study assumed that the random capacity levels of each component were mutually statistically independent, it is possible to include dependence in the model by adding another source of randomness that is not decision-dependent. Let $\otherrandvar$ (with support $\Gamma$) denote this other source of randomness, which, in the context of network or facility failures, could represent the location and severity of a weather event. Then, we model the realizations of the capacity levels of components as mutually independent \textit{conditional on the realization of} $\otherrandvar$; i.e., $\probability[\staterealizvar \,|\, \firststageallocvec,\otherrandvar]$. The formulation is as follows.

\begin{mini}|s|
    {\firststageallocvec\in \firststageallocfeasregion}{\costallocvec^T\firststageallocvec + \sum_{\otherrandvar\in \Gamma} \probability[\otherrandvar] \sum_{\staterealizvar\in \randvecsupport} \probability[\staterealizvar \,|\, \firststageallocvec,\otherrandvar]\,\recfn(\staterealizvar)}{\label{mod:first-stage-cond-indep}}{}
\end{mini}

While in this study we reformulated the problem so that the allocation variables $\firststageallocvec$ were binary, it could be possible to implement a version of the successive refinement algorithm for problems with continuous or integer allocation variables. For these problems the McCormick envelope would not be exact for the linear programming relaxation. Thus, a spatial branch-and-bound strategy would be necessary. A challenge would be to determine how to refine the partition tree dynamically within a spatial brand-and-bound algorithm.

\newpage
% BibTeX users please use one of
%\bibliographystyle{spbasic}      % basic style, author-year citations
\bibliographystyle{unsrt}      % mathematics and physical sciences
\bibliography{ddu}   % name your BibTeX data base

\begin{thebibliography}{10}

\bibitem{Zhan.Pinson.2017}
Yiduo Zhan, Qipeng~P. Zheng, Jianhui Wang, and Pierre Pinson.
\newblock {Generation Expansion Planning With Large Amounts of Wind Power via Decision-Dependent Stochastic Programming}.
\newblock {\em IEEE Transactions on Power Systems}, 32(4):3015--3026, 2017.

\bibitem{karaesmen2004overbooking}
Itir Karaesmen and Garrett Van~Ryzin.
\newblock Overbooking with substitutable inventory classes.
\newblock {\em Operations Research}, 52(1):83--104, 2004.

\bibitem{Medal.2015}
Hugh~R. Medal, Edward~A. Pohl, and Manuel~D. Rossetti.
\newblock {Allocating Protection Resources to Facilities When the Effect of Protection is Uncertain}.
\newblock {\em IIE Transactions}, 48(3):220--234, 2015.

\bibitem{Goel.Grossmann.2006}
Vikas Goel and Ignacio~E. Grossmann.
\newblock {A Class of stochastic programs with decision dependent uncertainty}.
\newblock {\em Mathematical Programming}, 108(2-3):355--394, 2006.

\bibitem{Goel.Mulkay.2006}
Vikas Goel, Ignacio~E. Grossmann, Amr~S. El-Bakry, and Eric~L. Mulkay.
\newblock {A novel branch and bound algorithm for optimal development of gas fields under uncertainty in reserves}.
\newblock {\em Computers \& Chemical Engineering}, 30(6-7):1076--1092, 2006.

\bibitem{Colvin.Maravelias.2009}
Matthew Colvin and Christos~T. Maravelias.
\newblock {A Branch and Cut Framework for Multi-Stage Stochastic Programming Problems Under Endogenous Uncertainty}.
\newblock {\em Computer Aided Chemical Engineering}, 27:255--260, 2009.

\bibitem{Luo.Mehrotra.2020}
Fengqiao Luo and Sanjay Mehrotra.
\newblock {Distributionally robust optimization with decision dependent ambiguity sets}.
\newblock {\em Optimization Letters}, 14(8):2565--2594, 2020.

\bibitem{Ma.Guo.2017}
Shanshan Ma, Liu Su, Zhaoyu Wang, Feng Qiu, and Ge~Guo.
\newblock {Resilience Enhancement of Distribution Grids Against Extreme Weather Events}.
\newblock {\em IEEE Transactions on Power Systems}, 33(5):4842--4853, 2017.

\bibitem{Zhang.Cao.2023}
Weixin Zhang, Changzheng Shao, Bo~Hu, Kaigui Xie, Pierluigi Siano, Mushui Li, and Maosen Cao.
\newblock {Transmission Defense Hardening Against Typhoon Disasters Under Decision-Dependent Uncertainty}.
\newblock {\em IEEE Transactions on Power Systems}, 38(3):2653--2665, 2023.

\bibitem{Yin.Hou.2023q6f}
Wenqian Yin, Yujia Li, Jiazuo Hou, Miao Miao, and Yunhe Hou.
\newblock {Coordinated Planning of Wind Power Generation and Energy Storage With Decision-Dependent Uncertainty Induced by Spatial Correlation}.
\newblock {\em IEEE Systems Journal}, 17(2):2247--2258, 2023.

\bibitem{Hu.Peng.2020}
Bo~Hu, Congcong Pan, Changzheng Shao, Kaigui Xie, Tao Niu, Chunyan Li, and Lvbin Peng.
\newblock {Decision-Dependent Uncertainty Modeling in Power System Operational Reliability Evaluations}.
\newblock {\em IEEE Transactions on Power Systems}, 36(6):5708--5721, 2020.

\bibitem{Losada.2012}
Chaya Losada, M.~Paola Scaparra, Richard~L. Church, and Mark~S. Daskin.
\newblock {The stochastic interdiction median problem with disruption intensity levels}.
\newblock {\em Annals of Operations Research}, 201(1):345--365, 2012.

\bibitem{OHanley.2013}
Jesse~R. O’Hanley, M.~Paola Scaparra, and Sergio García.
\newblock {Probability chains: A general linearization technique for modeling reliability in facility location and related problems}.
\newblock {\em European Journal of Operational Research}, 230(1):63--75, 2013.

\bibitem{Zhou.2022}
Rui Zhou, Tanveer~Hossain Bhuiyan, Hugh~R. Medal, Michael~D. Sherwin, and Dong Yang.
\newblock {A stochastic programming model with endogenous uncertainty for selecting supplier development programs to proactively mitigate supplier risk}.
\newblock {\em Omega}, 107:102542, 2022.

\bibitem{Bhuiyan.2019}
Tanveer~Hossain Bhuiyan, Maxwell~C. Moseley, Hugh~R. Medal, Eghbal Rashidi, and Robert~K. Grala.
\newblock {A stochastic programming model with endogenous uncertainty for incentivizing fuel reduction treatment under uncertain landowner behavior}.
\newblock {\em European Journal of Operational Research}, 277(2):699--718, 2019.

\bibitem{Bhuiyan.2020}
Tanveer~Hossain Bhuiyan, Hugh~R. Medal, and Sarah Harun.
\newblock {A stochastic programming model with endogenous and exogenous uncertainty for reliable network design under random disruption}.
\newblock {\em European Journal of Operational Research}, 285(2):670--694, 2020.

\bibitem{Yin.Hou.2023}
Wenqian Yin, Shuanglei Feng, and Yunhe Hou.
\newblock {Stochastic Wind Farm Expansion Planning With Decision-Dependent Uncertainty Under Spatial Smoothing Effect}.
\newblock {\em IEEE Transactions on Power Systems}, 38(3):2845--2857, 2023.

\bibitem{Peeta.Viswanath.2010}
Srinivas Peeta, F.~Sibel Salman, Dilek Gunnec, and Kannan Viswanath.
\newblock {Pre-disaster investment decisions for strengthening a highway network}.
\newblock {\em Computers \& Operations Research}, 37(10):1708--1719, 2010.

\bibitem{Held.Woodruff.2005}
Harald Held and David~L. Woodruff.
\newblock {Heuristics for Multi-Stage Interdiction of Stochastic Networks}.
\newblock {\em Journal of Heuristics}, 11(5-6):483--500, 2005.

\bibitem{Du.Peeta.2014}
Lili Du and Srinivas Peeta.
\newblock {A Stochastic Optimization Model to Reduce Expected Post-Disaster Response Time Through Pre-Disaster Investment Decisions}.
\newblock {\em Networks and Spatial Economics}, 14(2):271--295, 2014.

\bibitem{10.1016/j.ejor.2019.06.024}
J.~Cole Smith and Yongjia Song.
\newblock {A survey of network interdiction models and algorithms}.
\newblock {\em European Journal of Operational Research}, 283(3):797--811, 2020.

\bibitem{10.1016/j.cie.2021.107708}
Mahdieh Mirzaei, S.~Mohammad J.~Mirzapour Al-e hashem, and Mohsen~Akbarpour Shirazi.
\newblock {A maximum-flow network interdiction problem in an uncertain environment under information asymmetry condition: Application to smuggling goods}.
\newblock {\em Computers \& Industrial Engineering}, 162:107708, 2021.

\bibitem{10.1016/j.cor.2006.09.019}
Maria~P. Scaparra and Richard~L. Church.
\newblock {A bilevel mixed-integer program for critical infrastructure protection planning}.
\newblock {\em Computers \& Operations Research}, 35(6):1905--1923, 2008.

\bibitem{Cormican.98}
Kelly~J Cormican, David~P Morton, and R~Kevin Wood.
\newblock {Stochastic Network Interdiction}.
\newblock {\em Operations Research}, 46(2):184--197, 1998.

\bibitem{huang1977bounds}
CC~Huang, William~T Ziemba, and Aharon Ben-Tal.
\newblock Bounds on the expectation of a convex function of a random variable: With applications to stochastic programming.
\newblock {\em Operations Research}, 25(2):315--325, 1977.

\end{thebibliography}

\end{document}